\documentclass[12pt]{article}
\usepackage{amsfonts, amssymb, epsfig,subfigure}
\usepackage{mathrsfs}
\usepackage{stmaryrd}
\usepackage{graphicx,amsmath}
\usepackage{latexsym}
\usepackage{verbatim}
\usepackage{diagbox}
\usepackage{color}
\usepackage{graphicx}
 \usepackage{epstopdf}
 \usepackage{subfigure}
\newtheorem{theorem}{Theorem}[section]
\newtheorem{lemma}{Lemma}[section]

\newtheorem{expl}{Example}[section]

\newtheorem{thm}{Theorem}[section]

\newtheorem{rem}[thm]{Remark}

\makeatletter
\@addtoreset{equation}{section}

\newcommand{\beq}[1]{\begin{equation} \label{#1}}
\newcommand{\eeq}{\end{equation}}
\newcommand{\bed}{\begin{displaymath}}
\newcommand{\eed}{\end{displaymath}}
\newcommand{\bea}{\bed\begin{array}{rl}}
\newcommand{\eea}{\end{array}\eed}

\newcommand{\barray}{\begin{array}{ll}}
\newcommand{\earray}{\end{array}}
\newcommand{\lf}{\lfloor}

\newcommand{\Spvek}[2][r]{%
  \gdef\@VORNE{1}
  \left(\hskip-\arraycolsep%
    \begin{array}{#1}\vekSp@lten{#2}\end{array}%
  \hskip-\arraycolsep\right)}

\def\vekSp@lten#1{\xvekSp@lten#1;vekL@stLine;}
\def\vekL@stLine{vekL@stLine}
\def\xvekSp@lten#1;{\def\temp{#1}%
  \ifx\temp\vekL@stLine
  \else
    \ifnum\@VORNE=1\gdef\@VORNE{0}
    \else\@arraycr\fi%
    #1%
    \expandafter\xvekSp@lten
  \fi}
\allowdisplaybreaks[2]
\makeatletter
  \newcommand\figcaption{\def\@captype{figure}\caption}
  \newcommand\tabcaption{\def\@captype{table}\caption}
\makeatother

\def\sqr#1#2{{\vcenter{\vbox{\hrule height.#2pt
\hbox{\vrule width.#2pt height#1pt \kern#1pt
\vrule width.#2pt} \hrule height.#2pt}}}}

\newcommand{\E}{\mathbb{E}}
\newcommand{\PP}{\mathbb{P}}
\newcommand{\RR}{\mathbb{R}}

\def\br{\breve}

\def\tr{\triangle} \def\lf{\left} \def\rt{\right}

\def\T{\tau}

\def\<{\langle} \def\>{\rangle}

\def\1{\oslash} \def\2{\oplus} \def\3{\otimes} \def\4{\ominus}
\def\5{\circ} \def\6{\odot} \def\7{\backslash} \def\8{\infty}
\def\9{\bigcap} \def\0{\bigcup} \def\+{\pm} \def\-{\mp}
\def\[{\langle} \def\]{\rangle}

\newcommand{\dis}{\displaystyle}

\def\nn{\nonumber}

\def\bc{\begin{center}}       \def\ec{\end{center}}
\def\ba{\begin{array}}        \def\ea{\end{array}}
\def\be{\begin{equation}}     \def\ee{\end{equation}}
\def\bea{\begin{eqnarray}}    \def\eea{\end{eqnarray}}
\def\beaa{\begin{eqnarray*}}  \def\eeaa{\end{eqnarray*}}
\def\la{\label}
\begin{document}
\title{  The Strong Convergence and  Stability of Explicit
Approximations for Nonlinear Stochastic  Delay  Differential Equations
}

\author{Guoting Song\thanks{School of Mathematics and Statistics,
Northeast Normal University, Changchun, 130024, China. }
\and  Junhao Hu\thanks{School of Mathematics and Statistics,
South-Central University for Nationalities, Wuhan, 430074, China. Research
of this author was supported  by National Natural Science Foundation of
China (61876192). }
\and Shuaibin Gao\thanks{School of Mathematics and Statistics,
South-Central University for Nationalities, Wuhan, 430074, China.}
\and Xiaoyue Li\thanks{School of Mathematics and Statistics,
Northeast Normal University, Changchun,  130024, China. Research
of this author was supported  by National Natural Science Foundation of
China (11971096) and the Fundamental Research Funds for the
Central Universities.}
\date{}}

\maketitle

\begin{abstract}
This paper focuses on  explicit
approximations for  nonlinear stochastic  delay  differential equations (SDDEs). Under the weakly local Lipschitz and some suitable conditions, a generic truncated Euler-Maruyama (TEM) scheme for SDDEs is proposed, which numerical solutions are
 bounded and converge to the exact solutions in $q$th moment for $q>0$. Furthermore, the $ {1}/{2}$ order convergent rate is yielded.
Under  the  Khasminskii-type condition, a more precise  TEM scheme is given, which numerical solutions are
   exponential
 stable in mean square and $\PP-1$. Finally,  several  numerical experiments  are carried out to illustrate our results.
\end{abstract}

 \vspace{3mm}

\section{Introduction}\label{Tntr}

This paper considers a stochastic delay differential equation (SDDE) described by
\begin{align}\la{s2.2}
\left\{
\begin{array}{ll}
\mathrm{d}x(t)=f(x(t),x(t-\tau))\mathrm{d}t +g(x(t),x(t-\tau))\mathrm{d}W(t),~~~t>0,&\\
~~~\!\!\!x(t)=\xi(t),~~~t\in[-\tau,~0],&\\
\end{array}
\right.
\end{align}
 where $\tau>0$ is a constant,  $f:  \RR^d \times \RR^d \rightarrow \RR^{d}$, and $g: \RR^d\times \RR^d  \rightarrow \RR^{d\times m}$. And $W(t)=( W_1(t),~W_2(t),~\cdots,~W_m(t) )^T$ is an m-dimensional
   Brownian motion in  the given complete probability space  $( \Omega,~\mathcal{F}, ~\PP )$  and   ${\mathcal{F}}_{t}$ is a natural filtration
 satisfying the usual conditions (that is, it is increasing and right
  continuous while ${\mathcal{F}}_{0}$ contains all $\mathbb{P}$-null sets).
 The SDDE  models play  a key role in communications, finance,
 medical sciences, ecology, and many other branches of industry and  science
  (see, e.g. \cite{Arriojas2007,Baker2000,Buckwar2000,Eurich1996,Mao2007,Mao2005,mao-yuan}). However,
 explicit solutions can hardly be obtained for  SDDEs and hence it is  necessary and significant to  develop their numerical methods.

 In fact, numerical
 methods  of SDDEs have attracted a lot of attentions.  Due to the easy implementation  explicit schemes have
  been established (see e.g. \cite{Baker2000,
  Bao2016,Buckwar2000,Dareiotis-Kumar-Sabanis,Guo2017,Sabanis2013,Kumar-Sabanis,
  Huang2020,Mao2003,mao-yuan,Niu2013,Song2019}), such as the  Euler-Maruyama (EM) scheme (see e.g.\cite{Baker2000,
  Bao2016,Buckwar2000,Sabanis2013,Kumar-Sabanis,Mao2003,mao-yuan}), the truncated EM scheme \cite{Guo2017}, the truncated
  Milstein scheme \cite{Song2019}, the projected EM scheme \cite{Huang2020},
 and the tamed Euler scheme \cite{Dareiotis-Kumar-Sabanis}. Since  implicit schemes
  sometimes achieve the better convergence rate,  some concentrated effort have
  been made into the implicit schemes (see e.g. \cite{Huang2014,Huang2012,Wang2011}).
  However, to the best of our knowledge,  most of the results  on the strong convergence rate of numerical solutions for  nonlinear SDDE (\ref{s2.2}) requires that $f$ and $g$ obey  the one-side Lipschitz condition
  $$2\langle x-\bar x ,f(x,y)-f(\bar x,\bar y)\rangle+|g(x,y)-g(\bar x,\bar y)|^2\leq C(|x-\bar x|^2+|y-\bar y|^2),$$
  where $x,~\bar x,~y,~\bar y\in\RR^d$ and $C$ is a constant. Although
a  kind of nonlinear SDDEs satisfies this condition,  a large kind of SDDEs is unavailable for it.
  For an example, consider the  scalar SDDE
  \begin{align}\la{intr_exp1}
\left\{
\begin{array}{ll}
\mathrm{d}x(t)=\left(x(t)-8x^{3}(t)
\right)\mathrm{d}t
+|x(t-1)|^{\frac{3}{2}}\mathrm{d}W(t),~~~t>0,&\\
~~\!x(t)=t^2,~~~t\in[-1, ~0].&\\
\end{array}
\right.
\end{align}
By computation, one notices
\begin{align*}
&2 \langle  x-\bar x ,8x^{3}-8\bar{x}^{3} \rangle+||y|^{\frac{3}{2}}-|\bar{y}|^{\frac{3}{2}}|^2\nn\\
=& 2(x-\bar x)^2-16(x-\bar x)^2(x^2+x\bar x+\bar x^2)+(|y|^{\frac{1}{2}}-|\bar{y}|^{\frac{1}{2}})^2 (|y|+|y|^{\frac{1}{2}} |\bar{y}|^{\frac{1}{2}}+|\bar{y}|)^2,
\end{align*}
which implies that the one-side Lipschitz condition doesn't hold for SDDE (\ref{intr_exp1}).
Guo-Mao-Yue in \cite{Guo2017} proposed a  truncated EM scheme to approximate SDDE (\ref{intr_exp1}), and yielded  the  mean square convergence rate, which is less than $ {1}/{2}$. 
Dareiotis-Kumar-Sabanis in \cite{Dareiotis-Kumar-Sabanis} gave the tamed Euler scheme for SDDE (\ref{intr_exp1}) and its convergence rate can achieve to $ {1}/{2}$  at some special time $T$. 
 For such kind of SDDEs without one-side Lipschitz condition,  to establish an appropriate numerical scheme and to estimate the $L^q$-convergence rate   in any time interval is still open for $q>2$.

On the other hand,  the stability of such SDDEs  is one of the major concerns in stochastic processes, systems theory and control \cite{Mao2007}. Especially,{
Mao-Rassias in \cite{Mao2005} established the  exponential moment stability  for such SDDEs under the local Lipschitz condition plus   the   Khasminskii-type condition
\be\la{khas}\mathcal{L}U(x,y)\leq -c_1U(x)+c_2U(y)-c_3V(x)+c_4V(y),\ee
where $U(\cdot)$ is a nonnegative continuously twice
 differentiable function on $\RR^d$, $V(\cdot)$ is a nonnegative continuously function on $\RR^d$,   the operator $\mathcal{L} $ is defined by \eqref{s2.1}, and   constants $c_i, ~i=1,\dots, 4$ are positive with certain restrictions. Li-Mao in \cite{li_mao2012} provided us with a criterion on the exponentially almost sure stability of the exact solution for such SDDEs.

 According to the  requirement of numerical experiments and simulations  the stability of the numerical solutions for SDDEs attracts much attention. Wu-Mao-Szpruch in \cite{wu_mao_lukas} gave a counterexample that the EM scheme can't reproduce the exponentially almost sure  stability for a  nonlinear SDDE while  the Backward EM (BEM) scheme can. 
Zhao-Yi-Xu in \cite{zhao_yi_xu} proved that the implicit split-step theta (SSD) method  preserves the exponential mean square stability under the Khasminskii-type condition for $\theta\in(\frac{1}{2},1]$. Nevertheless, it is known that more computational efforts and cost are required using
the implicit equation in each iteration. Thus easily implementable explicit methods for nonlinear
SDDEs are more desirable in order to capture the stability, which motivated the recent development of modified EM methods. {Cong-Zhan-Guo in \cite{cong-zhan-guo} proposed the partially truncated Euler-Maruyama method which reproduces the  almost sure exponentially
stability of the exact solution for  SDDEs with Markovian Switching under \eqref{khas} with $V(\cdot)\equiv0$.}
Although the various stable numerical methods are investigated well to design the explicit scheme keeping the stability for nonlinear SDDEs under the flexible Khasminskii-type condition \eqref{khas} remains to unsolved.
 Hence, to establish  an easy implementable numerical scheme capturing the stability of SDDEs is the other main aim.

 Motivated by the above works,  borrowing the ideas from \cite{Li2018} we  develop the  explicit truncated  numerical scheme  to approximate nonlinear SDDEs. Under the polynomial growth coefficient conditions the  $ {1}/{2}$ order  rate of strong convergence  is yielded for the TEM scheme. Moreover, a more precise TEM scheme is constructed, which numerical solutions realize  the underlying exponential  stability under the flexible Khasminskii-type condition. Some simulations are carried out to check the effectiveness  of the TEM schemes.

This paper is organized in the following way. Section \ref{NP} gives some notations and preliminary results
with respect to the exact solution for SDDE (\ref{s2.2}).
Section \ref{MR}  lists the main results, including the  convergence, the convergence rate and the stability.
 Section \ref{NE}  gives two examples and  the corresponding simulations to illustrate the main  results.
 Section \ref{Con}  concludes the paper.}
\section{Notations and preliminary results}\label{NP}
We firstly present some standard notations and definitions which are
necessary for further consideration. The norm of a vector $x\in \RR^d$ and
the Hilbert-Schmidt norm of a matrix $A\in \RR^{d\times m}$ are
respectively denoted by $|x|$ and $|A|$. The transpose of a vector $x\in \RR^d$ is denoted
by $x^T$ and the inner product of two vectors $x,y\in\RR^d$ is denoted
by $\langle x,y\rangle=x^Ty$. Let $[a]$ denote the integer part of the real
number $a$. For two real numbers a and b, let $a\vee{}b=\max(a,b)$ and $a\wedge{}b=\min(a,b)$.
Let $\mathbb{R}_{+}=[0,\infty)$ and $\tau>0$.
 By $\mathcal{C}([-\tau,~0];~\mathbb{R}^{d})$, we denote the  space of all continuous  $\mathbb{R}^{d}$-valued functions defined  on
$[-\tau,~0]$ equipped with the supremum norm $\|\xi\|
=\sup_{-\tau\leq\theta\leq0}|\xi(\theta)|$.
By $\mathcal{C}(\RR^d;   ~\RR_+)$,
 we denote  the space of all  continuous nonnegative functions defined on $\RR^d$.
  By $\mathcal{V}(\RR^d\times\RR^d;   ~\RR_+)$, we denote the space of all nonnegative functions $\hat V(x,y)$  defined on $\RR^d\times\RR^d$ satisfying $\hat V(x,x)=0$.
Moreover, denote by
   $\mathcal{C}^{2} (\RR^d;  ~{{\RR}}_+)$ the space of all continuously twice
 differentiable  nonnegative functions
  defined on $\RR^d$.
 If $U \in \mathcal{C}^{2} (\RR^d;   ~\RR_+)$, define an
operator ${\cal{L}}U$
 $:  \RR^d \times \RR^d\to  \RR$ by
\begin{align}\label{s2.1}
{\cal{L}}U(x,y) =&\langle f(x,y),U_x(x)\rangle
 +\frac{1}{2} \left\langle g(x,y),U_{xx}(x)g(x,y)\right\rangle.
\end{align}
For any set $A$,~$ \boldsymbol{1}_A(x)=1$ if $x\in A$  otherwise 0.
Let $\delta_1,~\delta_2$ be two  $\mathcal{F}_t$-stopping times
 with $\delta_1\leq\delta_2$ \hbox{a.s}, then define the stochastic interval
 $$[[\delta_1,\delta_2]]=\{(t,\omega)\in \RR_+\times \Omega:\delta_1\leq t\leq \delta_2\}.$$
Denote a generic positive constant by $C$  which value may vary in different appearance.

We impose the following hypotheses.

\textbf{(H1) } (the weakly  local   Lipschitz condition) For any  $l_1 >0$, there exists a positive
constant $L_{l_1} $ such that, for any $x,~\bar{x},~y\in \RR ^d  $
with  $|x |\vee|\bar{x}|\vee|{y}|\leq l_1$,
\begin{align*}
  |f(x, y)-f(\bar{x}, y)|\vee |g(x, y)-g(\bar{x}, y)|\leq{}
  L_{l_1} |x-\bar x|.
  \end{align*}

  \textbf{ (H2) }(the  Khasminskii-type condition) There exist constants $q>0,~K_1\geq0, ~K_2\geq0$  as well as a function $V_1 \in\mathcal{C}(\RR^d;   ~\RR_+)$ such that
\begin{align}\label{s2.4}
&(1+|x|^{2})^{\frac{q}{2}-1}\Big(\big\langle2x ,f(x,y)\big\rangle+\big((q-1)\vee1\big)|g(x,y)|^2\Big)\nn\\
&\leq{}K_1(1+|x|^{q}+|y|^{q})-K_2(V_1(x)-V_1(y))
,~~~~\forall~ x,~y  \in\RR^{d}.
\end{align}

\textbf{(H3)}
 For any given positive constant $M_1>0$,  functions $f(x,y)$ and $g(x,y)$ are uniformly continuous in the argument corresponding  $y$ for any  $x\in\RR^d$ satisfying $|x|\leq M_1$, that is , for any $x,~y,~\bar y\in \RR^d$ with $|x|\leq M_1$,
\begin{align*}
\lim_{y\rightarrow \bar y}\sup_{|x|\leq M_1}\big[|f(x,y)-f(x,\bar y)|+|g(x,y)-g(x,\bar y)|\big]=0.
\end{align*}

\begin{theorem}\la{th1}
  Let  
  $(\textup{H}1)$ and $(\textup{H}2)$ hold. Then  SDDE \eqref{s2.2} with an initial data $\xi\in{}\mathcal{C}([-\tau,~0]; ~\RR^d)$ has a unique global solution $x(t)$ satisfying
   \be\la{s2.5}\sup_{0\leq t\leq T}
   \E|x(t)|^{q  } \leq C,~~~~~~\forall ~T>0.
   \ee
Furthermore, for any  constant $M_2>\|\xi\|$, let
 \begin{align}\label{sM_1}
 \vartheta_{M_2}=\inf\left\{ t\geq -\tau:|x(t)|\geq M_2\right\}.
 \end{align}
 Then we obtain
\begin{align} \la{s2.6}
\mathbb{P}\{\vartheta_{M_2}\leq T\}\leq \frac{C}{M_2^{q}}.
\end{align}
\end{theorem}
\begin{proof}\textbf {Proof.}
Fix a positive constant $l$, it follows from  (\ref{s2.4})  that for any $x,~y\in\RR^d$ with $|y|\leq l$,
 \begin{align}\label{SONG1}
 &\big\langle2x ,f(x,y)\big\rangle+|g(x,y)|^2\nn\\
 \leq&\frac{1}{(1+|x|^{2})^{\frac{q}{2}-1}}[K_1(1+|x|^{q}+|y|^{q})-K_2(V_1(x)-V_1(y))]\nn\\
\leq&2K_1(1+|x|^2)+( l^qK_1+\max_{|y|\leq l}V_1(y)K_2)(1+|x|^2)\leq C(l)(1+|x|^2).
 \end{align}
Under (H1) and (\ref{SONG1}), due to \cite[Theorem 2.1]{Sabanis2013}  SDDE (\ref{s2.2}) admits   a unique global solution with the initial data
$\xi\in{}\mathcal{C}([-\tau,~0]; ~\RR^d)$.
Let $U(x)=\big(1+|x|^2\big)^{\frac{q}{2}}$,
   where  $q$ is given in  (H2). Due to  (\ref{s2.4}) we compute
 \begin{align}\label{s2.7}
 &{\cal{L}}U\Big(x(t),x\big(t-\tau\big)\Big)\nn\\
=&\frac{q}{2}\big(1+|x(t)|^{2}\big)^{\frac{q}{2}-2}
\Bigg[\big(1+|x(t)|^{2}\big)\Big(\big\langle2x(t) ,f\big(x(t),x(t-\tau)\big)\big\rangle\nn\\
&+\big| g\big(x(t),x(t-\tau)\big)\big|^2\Big)+\big(q-2\big)\big|\big\langle x(t),g\big(x(t),x(t-\tau)\big)\big\rangle \big|^{2}\Bigg]\nn\\
 \leq&\frac{q}{2}\big(1+|x(t)|^{2}\big)^{\frac{q}{2}-2}\Bigg[\big(1+|x(t)|^{2}\big)
 \Big(\big\langle2x(t) ,f\big(x(t),x(t-\tau)\big)\big\rangle\nn\\
 &+\big|g\big(x(t),x(t-\tau)\big)\big|^2\Big)+\big((q-2)\vee 0\big)\big|x(t)\big|^2\big|g\big(x(t),x(t-\tau)\big)\big|^2\Bigg]\nn\\
 \leq&\frac{q}{2}\big(1+|x(t)|^{2}\big)^{\frac{q}{2}-1}\Big(\big\langle2x(t) ,f\big(x(t),x(t-\tau)\big)\big\rangle\nn\\
 &+((q-1)\vee 1)\big|g\big(x(t),x(t-\tau)\big)\big|^2\Big)\nn\\
 \leq& \frac{q}{2}K_1\Big(1+\big|x(t)\big|^{q}+\big|x(t-\tau)\big|^{q}\Big)
 -\frac{q}{2}K_2\Big(V_1\big(x(t)\big)-V_1\big(x(t-\tau)\big)\Big).
 \end{align}
By \cite[Theorem 3.1]{Mao2005} together with the definition of $U$, it yields
\begin{align*}
&\sup_{0\leq t\leq T}\E\big(1+|x(t)|^2\big)^{\frac{q}{2}}\nn\\
&\leq \bigg(U\big(\xi(0)\big)+\frac{q}{2}\int^{0}_{-\tau}\Big[K_1U\big(\xi(s)\big)
+K_2V_1\big(\xi(s)\big)\Big]ds+\frac{q}{2}K_1T\bigg) e^{qK_1T}=:C.
\end{align*}
Due to (\ref{s2.7}) and  using Dynkin's formula we  get that,  for any $0\leq t\leq T$,
 \begin{align*}
 &\E\big(1+|x(t\wedge\vartheta_{M_2})|^2\big)^{\frac{q}{2}}\nn\\
 \leq&\big(1+|\xi(0)|^2\big)^{\frac{q}{2}}
 +\frac{q}{2}\E\int_{0}^{t\wedge\vartheta_{M_2}}
 \bigg[K_1\Big(1+\big(1+|x(s)|^2\big)^{\frac{q}{2}}\nn\\
 &+\big(1+|x(s-\tau)|^2\big)^{\frac{q}{2}}\Big)
 -K_2\big(V_1\big(x(s)\big)-V_1\big(x(s-\tau)\big)\big)\bigg]\mathrm ds\nn\\
 \leq&\big(1+|\xi(0)|^2\big)^{\frac{q}{2}}+\frac{q}{2}K_1T
 + qK_1\E\int_{0}^{t\wedge\vartheta_{M_2}}
\big(1+|x(s)|^2\big)^{\frac{q}{2}}\mathrm ds\nn\\
 &+\frac{q}{2}K_1\int_{-\tau}^{0}\big(1+|\xi(s)|^2\big)^{\frac{q}{2}}\mathrm ds
-\frac{q}{2}K_2\E\int_{0}^{t\wedge\vartheta_{M_2}}V_1\big(x(s)\big) \mathrm ds\nn\\
 &+\frac{q}{2}K_2\E\int_{0}^{t\wedge\vartheta_{M_2}}V_1\big(x(s)\big) \mathrm ds  +\frac{q}{2}K_2\int_{-\tau}^{0}V_1\big(\xi(s)\big)\mathrm ds \nn\\
 \leq &\big(1+|\xi(0)|^2\big)^{\frac{q}{2}}+\frac{q}{2}K_1T
 +qK_1\E\int_{0}^{t\wedge\vartheta_{M_2}}\big(1+|x(s)|^2\big)^{\frac{q}{2}}\mathrm ds\nn\\
 &+\frac{q}{2}K_1\int_{-\tau}^{0}\big(1+|\xi(s)|^2\big)^{\frac{q}{2}}\mathrm ds+\frac{q}{2}K_2\int_{-\tau}^{0}V_1\big(\xi(s)\big) \mathrm ds\nn\\
 =:&C_1+qK_1\E\int_{0}^{t\wedge\vartheta_{M_2}}\big(1+|x(s)|^2\big)^{\frac{q}{2}}\mathrm ds,
 \end{align*}
 which implies
 \begin{align*}
 &\sup_{0\leq t\leq T}\E\big(1+|x(t\wedge\vartheta_{M_2})|^2\big)^{\frac{q}{2}}\nn\\
 \leq& C_1+qK_1\int_{0}^{T}\sup_{0\leq s\leq t}
 \E\big(1+|x(s\wedge\vartheta_{M_2})|^2\big)^{\frac{q}{2}}\mathrm dt.
 \end{align*}
Applying the Gronwall inequality \cite[p.45, Theorem 8.1]{Mao2007} yields that
\begin{align*}
\sup_{0\leq t\leq T}\E\big(1+|x(t\wedge\vartheta_{M_2})|^2\big)^{\frac{q}{2}}
\leq C_1e^{qK_1T}.
\end{align*}
Thus
\begin{align*}
\mathbb{P}\{\vartheta_{M_2}\leq T\}M_2^q
\leq\E\Big[|x(\vartheta_{M_2})|^q\boldsymbol{1}_{\{\vartheta_{M_2}\leq T\}}\Big]
&\leq\E\big(1+|x(T\wedge\vartheta_{M_2})|^2\big)^{\frac{q}{2}}\leq C.
\end{align*}
Then the required inequality \eqref{s2.6} follows.
\end{proof}

\section{Main results}\label{MR}
 In order to construct an appropriate numerical scheme, we firstly estimate the growth rate of coefficients.
 Under (H1) and (H3),  choose a strictly increasing continuous function $\Phi : [1, \infty)\rightarrow \RR_+$  satisfying
\begin{align}\la{s3.2}
\sup_{|x|\vee |y|\leq l} \dis\left(\frac{ |f (x,y)|}{1+|x|}\vee\frac{| g(x, y)|^2 }{(1+|x|)^2}\right)\leq{}\Phi(l),~~~\forall~l\geq1.
\end{align}
Let $\Phi^{-1} : [\Phi(1), \infty) \rightarrow \RR_{+}$ be the inverse function of $\Phi$. For any given  stepsize $\triangle\in(0,1]$, let
 \begin{align}\label{s3.4}
h_{\Phi, \mu}(\triangle )=K \triangle^{-\mu},
\end{align}
 where $K:=\Phi(\|\xi\|\vee 1)$ and $\mu\in(0,\frac{1}{2}]$.
 ~Define a truncation mapping $\Upsilon_{\Phi,\mu}^{\triangle }:\RR^d\rightarrow \RR^d¡¡$ by
\begin{align*}
\Upsilon_{\Phi,\mu}^{\triangle }(x)= \Big(|x|\wedge \Phi^{-1}\big( h_{\Phi,\mu}(\triangle)\big)\Big) \frac{x}{|x|},
\end{align*}
where $\frac{x}{|x|}=\mathbf{0}$ when $x=\mathbf{0}\in \mathbb{R}^d$.

Then the truncated Euler-Maruyama(TEM) scheme  SDDE (\ref{s2.2}) as follows:
Choose a positive integer $N$ such that $\tr=\frac{\T}{N}\in(0,~1]$.
Define $t_{i}=i\tr,~\forall~i\geq-N$.
 And define
 \begin{align}\la{s3.6}
\left\{
\begin{array}{lll}
z_{i}^\tr=\xi(i\tr), ~\forall ~i=-N,\ldots,0,&\\
\breve z_{i+1}^\tr=z_{i}^\tr+f(z_{i}^\tr,z_{i-N}^\tr)\tr +g(z_{i}^\tr, z_{i-N}^\tr)\triangle W_i,~\forall ~i=0,1,\ldots,& \\
 z_{i+1 }^\tr =\Upsilon_{\Phi,\mu}^{\triangle}(\breve  z_{i+1}^\tr),&\\
\end{array}
\right.
\end{align}
where $\triangle W_i=W(t_{i+1})-W(t_{i})$.  So this scheme prevents the diffusion term from producing extra-ordinary large value.
{ One observes that
\begin{align}\la{s3.5}
|f (z_{i}^\tr, z_{i-N}^\tr)| \leq h_{\Phi,\mu}(\tr )  (1+| z_{i}^\tr|), ~~~|g(z_{i}^\tr, z_{i-N}^\tr)| \leq    h^{\frac{1}{2}}_{\Phi,\mu}(\tr) (1+| z_{i}^\tr|).
\end{align}
Define two continuous-time numerical schemes $\breve  z_{\tr}(t),~z_{\tr}(t)$ by
 \begin{align}\la{s3.7}
\breve  z_{\tr}(t):=\breve  z_{i}^\tr,~~~~z_{\tr}(t):=z_{i}^\tr,~~~~\forall t\in[t_i,  t_{i+1}).
\end{align}}
\subsection{Moment boundedness}
To study the convergence of the  TEM scheme (\ref{s3.6}), we need to get the $q$th moment boundedness  of the TEM scheme \eqref{s3.6}.
  \begin{theorem}\la{th3}
Assume that $(\textup{H}1)$-$(\textup{H}3)$ hold. Then the TEM scheme \eqref{s3.6} has the following property
  \be\la{Y_1}
  \sup_{0<\tr\leq 1}\sup_{0\leq i\tr\leq T}
   \E|z_{i}^\tr|^{q } \leq C,~~~~\forall ~T>0.
   \ee
  \end{theorem}
\begin{proof}\textbf{Proof.}
Define $f_i=f(z_{i}^\tr, z_{i-N}^\tr),~g_i=g(z_{i}^\tr, z_{i-N}^\tr)$ for short.
For any $T>0$ and $1\leq i\leq [T/\tr]$, one observes from (\ref{s3.6}) that
 \begin{align}\la{s4.1}
\!\!\!(1+  |\breve z_{i}^\tr|^2)^{\frac{q}{2}}
   =&\Big[1+ |z_{i-1}^\tr+f_{i-1}\tr+ g_{i-1}\tr W_{i-1}|^2
  \Big]^{\frac{q}{2}}\nn\\
  =& \left(1+ |z_{i-1}^\tr|^2 \right)^{\frac{q}{2}}
  ( 1+\Gamma_{i-1})^{\frac{q}{2}},
\end{align}
where
\begin{align*}
\Gamma_{i-1} =&\big(1+|z_{i-1}^\tr|^2\big)^{-1}\big[|f_{i-1}|^2\tr^2+|g_{i-1} \tr W_{i-1}|^2+2\langle z_{i-1}^\tr, f_{i-1}\rangle\tr\nn\\
&+2\langle z_{i-1}^\tr,  g_{i-1}\tr W_{i-1}\rangle
+2\langle f_{i-1},g_{i-1} \tr W_{i-1}\rangle\tr\big].
\end{align*}
From (\ref{s4.1}) one  observes that $\Gamma_{i-1}>-1$.
 For  the given constant  $q>0$, choose an integer $k$ such that $2k<q\leq 2(k+1)$. It follows from \cite[Lemma 3.3]{yang2018} and   (\ref{s4.1}) that
\begin{align}\la{s4.3}
 &\E \left[ \left(1+|\breve{z}_{i}^\tr |^2\right)^{\frac{q}{2}}\big|\mathcal{F}_{t_{i-1}}\right]\nn\\
   \leq&  (1 + |z_{i-1}^\tr|^2)^{\frac{q}{2}}\bigg[1 + \frac{q}{2} \E\big(\Gamma_{i-1}\big|\mathcal{F}_{t_{i-1}}\big)\nn\\
     &+ \frac{q(q-2)}{8} \E\big(\Gamma_{i-1}^2\big|\mathcal{F}_{t_{i-1}}\big)
  +  \E\big(\Gamma_{i-1}^3P_k(\Gamma_{i-1})\big|\mathcal{F}_{t_{i-1}}\big)\bigg],
\end{align}
where $P_k(\cdot)$ represents a $k$th-order polynomial  whose coefficients depend only on $q$.
Noticing that the increment $\tr W_{i-1}$ is independent of
$\mathcal{F}_{t_{i-1}}$. One  has for any $A\in\mathbb{R}^{d\times m}$
\begin{align}\la{s4.4}
&\E((A \tr W_{i-1})|\mathcal{F}_{t_{i-1}})=0,
~~~~\E(|A\tr W_{i-1}|^{2}|\mathcal{F}_{t_{i-1}})=|A|^2\tr.
\end{align}
Using  (\ref{s3.4}), (\ref{s3.5}) and (\ref{s4.4}),  we compute
\begin{align}\la{s4.5}
 &\E\big[\Gamma_{i-1}\big|\mathcal{F}_{t_{i-1}}\big]\nn\\
=&\dis {(1+|z_{i-1}^\tr|^2)^{-1 }} \Big(2\langle z_{i-1}^\tr, f_{i-1}\rangle \tr +|g_{i-1}|^2\tr
+|f_{i-1}|^2\tr^2 \Big)\nn\\
\leq &\dis {(1+|z_{i-1}^\tr|^2)^{-1 }}\Big(2\langle z_{i-1}^\tr, f_{i-1}\rangle\tr+|g_{i-1}|^2\tr
+h^2_{\Phi,\mu}(\tr)(1+|z_{i-1}^\tr|)^2\tr^2\Big)\nn\\
\leq & \dis {(1+|z_{i-1}^\tr|^2)^{-1 }}\Big(2\langle z_{i-1}^\tr, f_{i-1}\rangle+|g_{i-1}|^2\Big)\tr+C\tr.
\end{align}
To estimate   $ {q(q-2)}\E\big[\Gamma_{i-1}^2\big|\mathcal{F}_{t_{i-1}}\big]/{8}$, we consider two cases.

Case (I). If $0<q\leq 2$, then $ {q(q-2)}/{8}\leq 0$.
One observes,
 \begin{align}\la{ssg4.4}
&\E((A \tr W_{i-1})^{2j-1}|\mathcal{F}_{t_{i-1}})=0,\nn\\
&\E(|A\tr W_{i-1}|^{j}|\mathcal{F}_{t_{i-1}})\leq C\tr^{\frac{j}{2}}, ~\forall~A\in R^{d\times m},~j\geq2.
\end{align}
Using (\ref{s3.4}), (\ref{s3.5}) and (\ref{ssg4.4}), we have
\begin{align*}
&\E\big[\Gamma_{i-1}^2\big|\mathcal{F}_{t_{i-1}}\big]\nn\\
\geq &\big(1+|z_{i-1}^\tr|^2\big)^{-2}\E\Bigg\{\bigg[|2\langle z_{i-1}^\tr, g_{i-1} \tr W_{i-1}\rangle|^2
+4\langle z_{i-1}^\tr, g_{i-1}\tr W_{i-1}\rangle\nn\\
&\times \Big(|f_{i-1}|^2\tr^2+|g_{i-1} \tr W_{i-1}|^2
+2\langle z_{i-1}^\tr, f_{i-1}\rangle\tr\nn\\
&+2\langle f_{i-1}, g_{i-1} \tr W_{i-1}\rangle\tr \Big)\bigg]\bigg|\mathcal{F}_{t_{i-1}}\Bigg\}\nn\\
\geq&4\big(1+|z_{i-1}^\tr|^2\big)^{-2}\big|\langle z_{i-1}^\tr, g_{i-1}\rangle\big|^2\tr
-8\big(1+|z_{i-1}^\tr|^2\big)^{-2}|z_{i-1}^\tr||f_{i-1}||g_{i-1}|^2\tr^2\nn\\
\geq&4\big(1+|z_{i-1}^\tr|^2\big)^{-2}|\langle z_{i-1}^\tr,g_{i-1}\rangle|^2\tr-32K^2\tr^{2-2\mu}.
\end{align*}

Case (II). If $q>2$, then $\frac{q(q-2)}{8}>0$.  By the similar way as Case (I) we have
 \begin{align*}
&\E\big[\Gamma_{i-1}^2\big|\mathcal{F}_{t_{i-1}}\big]\nn\\
= &\big(1+|z_{i-1}^\tr|^2\big)^{-2}\E\Bigg\{\bigg[|2\langle z_{i-1}^\tr, g_{i-1} \tr W_{i-1}\rangle|^2
+4\langle z_{i-1}^\tr, g_{i-1}\tr W_{i-1}\rangle\nn\\
&\times\Big(|f_{i-1}|^2\tr^2+|g_{i-1} \tr W_{i-1}|^2
+2\langle z_{i-1}^\tr, f_{i-1}\rangle\tr+2\langle f_{i-1}, g_{i-1} \tr W_{i-1}\rangle\tr \Big)\nn\\
&+\Big(|f_{i-1}|^2\tr^2+|g_{i-1} \tr W_{i-1}|^2
+2\langle z_{i-1}^\tr, f_{i-1}\rangle\tr+2\langle f_{i-1}, g_{i-1} \tr W_{i-1}\rangle\tr \Big)^2
\bigg]\bigg|\mathcal{F}_{t_{i-1}}\Bigg\}\nn\\
\leq&4\big(1+|z_{i-1}^\tr|^2\big)^{-2}\big|\langle z_{i-1}^\tr, g_{i-1}\rangle\big|^2\tr
+\big(1+|z_{i-1}^\tr|^2\big)^{-2}\Big[8|z_{i-1}^\tr||f_{i-1}||g_{i-1}|^2\tr^2\nn\\
&+4|f_{i-1}|^4\tr^4+4|g_{i-1}|^4\tr^2+16|z_{i-1}^\tr|^2|f_{i-1}|^2\tr^2
+16|f_{i-1}|^2|g_{i-1}|^2\tr^3)\nn\\
\leq&4\big(1+|z_{i-1}^\tr|^2\big)^{-2}|\langle z_{i-1}^\tr, g_{i-1}\rangle|^2\tr\nn\\
&+\big(1+|z_{i-1}^\tr|^2\big)^{-2}\Big[8h^2_{\Phi,\mu}(\tr)|z_{i-1}^\tr|(1+|z_{i-1}^\tr|)^3\tr^2\nn\\
&+4h^4_{\Phi,\mu}(\tr)(1+|z_{i-1}^\tr|)^4\tr^4+4h^2_{\Phi,\mu}(\tr)(1+|z_{i-1}^\tr|)^4\tr^2\nn\\
&+16h^2_{\Phi,\mu}(\tr)|z_{i-1}^\tr|^2(1+|z_{i-1}^\tr|)^2\tr^2+16h^3_{\Phi,\mu}(\tr)(1+|z_{i-1}^\tr|)^4\tr^3\Big]\nn\\
\leq&4\big(1+|z_{i-1}^\tr|^2\big)^{-2}|\langle z_{i-1}^\tr,g_{i-1}\rangle|^2\tr
+C\tr.
\end{align*}
Combining both cases  implies  that
\begin{align}\label{s4.6}
\frac{q(q-2)}{8}\E\big[\Gamma_{i-1}^2\big|\mathcal{F}_{t_{i-1}}\big]\leq \frac{q(q-2)}{2}\big(1+|z_{i-1}^\tr|^2\big)^{-2}|\langle z_{i-1}^\tr,g_{i-1}\rangle|^2\tr
+C\tr.
\end{align}
To estimate $\E\big(\Gamma_{i-1}^3P_k(\Gamma_{i-1})\big|\mathcal{F}_{t_{i-1}}\big)$,
 we begin with $\E\big(\Gamma^3_{i-1}\big|\mathcal{F}_{t_{i-1}}\big)$.
Using (\ref{s3.4}), (\ref{s3.5}) and (\ref{ssg4.4})  we obtain
\begin{align*}
&\E\left[\Gamma^3_{i-1}\big|\mathcal{F}_{t_{i-1}}\right]\nn\\
=&(1+|z_{i-1}^\tr|^2)^{-3}\E\Big\{\big[(|f_{i-1}|^2\tr^2+|g_{i-1}\tr W_{i-1}|^2+2\langle z_{i-1}^\tr,f_{i-1}\rangle\tr)\nn\\
&+(2\langle z_{i-1}^\tr,g_{i-1}\tr W_{i-1}\rangle+2\langle f_{i-1},g_{i-1}\tr W_{i-1}\rangle\tr)\big]^3\big|\mathcal{F}_{t_{i-1}}
\Big\}\nn\\
=&(1+|z_{i-1}^\tr|^2)^{-3}\E\Big\{\big[(|f_{i-1}|^2\tr^2+|g_{i-1}\tr W_{i-1}|^2+ 2\langle z_{i-1}^\tr,f_{i-1}\rangle\tr)^3\nn\\
&+3(|f_{i-1}|^2\tr^2+|g_{i-1}\tr W_{i-1}|^2+2\langle z_{i-1}^\tr,f_{i-1}\rangle\tr)\nn\\
&\times(2\langle z_{i-1}^\tr,g_{i-1}\tr W_{i-1}\rangle+2\langle f_{i-1},g_{i-1}\tr W_{i-1}\rangle\tr)^2\big]\big|\mathcal{F}_{t_{i-1}}
\Big\}\nn\\
\geq&(1+|z_{i-1}^\tr|^2)^{-3}\E\Big\{\big[-8|z_{i-1}^\tr|^3|f_{i-1}|^3\tr^3-6(|f_{i-1}|^2\tr^2+|g_{i-1}\tr W_{i-1}|^2)^2|z_{i-1}^\tr||f_{i-1}|\tr\nn\\
&-6|z_{i-1}^\tr||f_{i-1}|\tr\times(2\langle z_{i-1}^\tr,g_{i-1}\tr W_{i-1}\rangle+2\langle f_{i-1},g_{i-1}\tr W_{i-1}\rangle\tr)^2\big]\big|\mathcal{F}_{t_{i-1}}
\Big\}\nn\\
\geq&-C(1+|z_{i-1}^\tr|^2)^{-3}(|z_{i-1}^\tr|^3|f_{i-1}|^3\tr^3+|z_{i-1}^\tr||f_{i-1}|^5\tr^5+|z_{i-1}^\tr||f_{i-1}||g_{i-1}|^4\tr^3\nn\\
&+|z_{i-1}^\tr|^3|f_{i-1}||g_{i-1}|^2\tr^2+|z_{i-1}^\tr||f_{i-1}|^3|g_{i-1}|^2\tr^4)\nn\\
\geq&-C(1+|z_{i-1}^\tr|^2)^{-3}\Big[h^3_{\Phi,\mu}(\tr)|z_{i-1}^\tr|^3(1+|z_{i-1}^\tr|)^3\tr^3+h^5_{\Phi,\mu}(\tr)|z_{i-1}^\tr|(1+|z_{i-1}^\tr|)^5\tr^5\nn\\
&+h^3_{\Phi,\mu}(\tr)|z_{i-1}^\tr|(1+|z_{i-1}^\tr|)^5\tr^3+h^2_{\Phi,\mu}(\tr)|z_{i-1}^\tr|^3(1+|z_{i-1}^\tr|)^3\tr^2\nn\\
&+h^4_{\Phi,\mu}(\tr)|z_{i-1}^\tr|(1+|z_{i-1}^\tr|)^5\tr^4\Big]\nn\\
\geq&-C(\tr^{3-3\mu}+\tr^{5-5\mu}+\tr^{3-3\mu}+\tr^{2-2\mu}+\tr^{4-4\mu})\geq-C\tr.
\end{align*}
On the other hand,
\begin{align*}
 &\E\left[\Gamma^3_{i-1}\big|\mathcal{F}_{t_{i-1}}\right]\nn\\
\leq& (1+|z_{i-1}^\tr|^2)^{-3}\Big[9(|f_{i-1}|^6\tr^6+|g_{i-1}|^6\tr^3+8|z_{i-1}^\tr|^3|f_{i-1}|^3\tr^3)\nn\\
&+24(|z_{i-1}^\tr|^2|f_{i-1}|^2|g_{i-1}|^2\tr^3+|f_{i-1}|^4|g_{i-1}|^2\tr^5+|z_{i-1}^\tr|^2|g_{i-1}|^4\tr^2\nn\\
&+|f_{i-1}|^2|g_{i-1}|^4\tr^4+2|z_{i-1}^\tr|^3|f_{i-1}||g_{i-1}|^2\tr^2+2|z_{i-1}^\tr||f_{i-1}|^3|g_{i-1}|^2\tr^4)\Big]\nn\\
\leq& C(1+|z_{i-1}^\tr|^2)^{-3}\Big[h^6_{\Phi,\mu}(\tr)(1+|z_{i-1}^\tr|)^6\tr^6+h^3_{\Phi,\mu}(\tr)(1+|z_{i-1}^\tr|)^6\tr^3\nn\\
&+h^3_{\Phi,\mu}(\tr)|z_{i-1}^\tr|^3(1+|z_{i-1}^\tr|)^3\tr^3+h^3_{\Phi,\mu}(\tr)|z_{i-1}^\tr|^2(1+|z_{i-1}^\tr|)^4\tr^3\nn\\
&+h^5_{\Phi,\mu}(\tr)(1+|z_{i-1}^\tr|)^6\tr^5+h^2_{\Phi,\mu}(\tr)|z_{i-1}^\tr|^2(1+|z_{i-1}^\tr|)^4\tr^2\nn\\
&+h^4_{\Phi,\mu}(\tr)(1+|z_{i-1}^\tr|)^6\tr^4+h^2_{\Phi,\mu}(\tr)|z_{i-1}^\tr|^3(1+|z_{i-1}^\tr|)^3\tr^2\nn\\
&+h^4_{\Phi,\mu}(\tr)|z_{i-1}^\tr|(1+|z_{i-1}^\tr|)^5\tr^4\Big]\leq C\tr.
\end{align*}
Thus, both of the above inequality imply $ \E\left[ a_0\Gamma^3_{i-1} \big|\mathcal{F}_{t_{i-1}}\right]\leq C\tr $ for any constant $a_0$, where $a_j$ represents the coefficient of $x^j$ term in polynomial $P_k(x)$.
We can also show that for any $j>3$
\begin{align*}
\E\left[|a_{j-3}\Gamma^j_{i-1}|\big|\mathcal{F}_{t_{i-1}}\right]
\leq&C(1+|z_{i-1}^\tr|^2)^{-j}\big(|f_{i-1}|^{2j}\tr^{2j}+|g_{i-1}|^{2j}\tr^j+|z_{i-1}^\tr|^j|f_{i-1}|^j\tr^j\nn\\
&+|z_{i-1}^\tr|^j|g_{i-1}|^j\tr^{\frac{j}{2}}+|f_{i-1}|^j|g_{i-1}|^j\tr^{\frac{3j}{2}})\nn\\
\leq&C\tr.
\end{align*}
These implies
\begin{align}\label{s4.7}
\E\big(\Gamma_{i-1}^3P_k(\Gamma_{i-1})\big|\mathcal{F}_{t_{i-1}}\big)\leq C\tr.
\end{align}
Subsituting (\ref{s4.5}), \eqref{s4.6}, (\ref{s4.7}) into (\ref{s4.3}) and using (H2), we  obtain
\begin{align}\la{Y_7}
&\E\Big[ (1+|\breve{z}_{i}^\tr |^2)^{\frac{q}{2}}\big|\mathcal{F}_{t_{i-1}}\Big]\nn\\
\leq  & \left(1+|z_{i-1}^\tr|^2\right)^{\frac{q}{2}}  \Big\{1+C\tr\nn\\
& +  \frac{q\tr}{2}\frac{(1+|z_{i-1}^\tr|^2)(2\langle z_{i-1}^\tr,f_{i-1}\rangle
 +|g_{i-1}|^2)
 +(q-2)|\langle z_{i-1}^\tr,g_{i-1}\rangle|^2}{(1+|z_{i-1}^\tr|^2)^{2}}  \Big\}\nn\\
 \leq  & \left(1
 +C\tr\right)\left(1+|z_{i-1}^\tr|^2\right)^{\frac{q}{2}}
 +\frac{q\tr}{2}(1+|z_{i-1}^\tr|^2)^{\frac{q}{2}-1}
 \big(2\langle z_{i-1}^\tr,f_{i-1}\rangle+((q-1)\vee 1)|g_{i-1}|^2\big)\nn\\
 \leq &  \left(1
 +C\tr\right)\left(1+|z_{i-1}^\tr|^2\right)^{\frac{q}{2}} +\frac{q}{2}\tr
 \Big[K_1\big(1+|z_{i-1}^\tr|^q+|z_{i-1-N}^\tr|^q\big)\nn\\
& -K_2\big(V_1(z_{i-1}^\tr)-V_1(z_{i-1-N}^\tr)\big)\Big]\nn\\
 \leq& \left(1
 +C\tr\right)\left(1+|z_{i-1}^\tr|^2\right)^{\frac{q}{2}}  +\frac{q}{2} K_1\tr 
\big(1+|z_{i-1-N}^\tr|^2\big)^{\frac{q}{2}}\nn\\
 &-\frac{q}{2}K_2\tr V_1(z_{i-1}^\tr)
 +\frac{q}{2} K_2\tr V_1(z_{i-1-N}^\tr).
\end{align}
Taking expectations on both sides of (\ref{Y_7}) yields
\begin{align}\la{s6}
&\E\big[(1+|z_{i}^\tr|^2)^{\frac{q}{2}}\big]-\E(1+|z_{i-1}^\tr|^2)^{\frac{q}{2}}\nn\\
\leq&\E\left\{\E\big[(1+|\breve{z}_{i}^\tr |^2)^\frac{q}{2}\big|\mathcal{F}_{t_{i-1}}\big]\right\}-\E(1+|z_{i-1}^\tr|^2)^{\frac{q}{2}}\nn\\
\leq& C\tr\E(1+|z_{i-1}^\tr|^2)^{\frac{q}{2}} +\frac{q}{2}K_1\tr
 \E\big(1+|z_{i-1-N}^\tr|^2\big)^{\frac{q}{2}}\nn\\
 &-\frac{q}{2}K_2\tr \E V_1(z_{i-1}^\tr)
+\frac{q}{2}K_2\tr \E V_1(z_{i-1-N}^\tr),
\end{align}
which implies
\begin{align*}
&\E\left[(1+|z_{i}^\tr|^2)^\frac{q}{2}\right]\nn\\
\leq&(1+|\xi(0)|^2)^{\frac{q}{2}}+C\tr\sum_{k=0}^{i-1}\E(1+|z_{k}^{\tr}|^2)^\frac{q}{2}
+\frac{q}{2}K_1\tr \sum_{k=0}^{i-1}\E(1+|z_{k-N}^{\tr}|^2)^\frac{q}{2}\nn\\
&-\frac{q}{2}K_2\tr \sum_{k=0}^{i-1}\E V_1(z_{k}^{\tr})
+\frac{q}{2}K_2\tr \sum_{k=0}^{i-1}\E V_1(z_{k-N}^{\tr})\nn\\
\leq&(1+|\xi(0)|^2)^{\frac{q}{2}}+C\tr\sum_{k=0}^{i-1}\E(1+|z_{k}^{\tr}|^2)^\frac{q}{2}
+\frac{Nq}{2}K_1\tr (1+\|\xi\|^2)^\frac{q}{2}\nn\\
&+\frac{q}{2}K_1\tr \sum_{k=0}^{(i-1-N)\vee 0}\E(1+|z_{k}^{\tr}|^2)^\frac{q}{2}
-\frac{q}{2}K_2\tr \sum_{k=0}^{i-1}\E V_1(z_{k}^{\tr})\nn\\
&+
\frac{Nq}{2}K_2\tr \max_{-N\leq j\leq0}V_1(\xi(t_j))
+\frac{q}{2}K_2\tr \sum_{k=0}^{(i-1-N)\vee 0}\E V_1(z_{k}^{\tr})\nn\\
\leq&C\tr\sum_{k=0}^{i-1}\E(1+|z_{k}^{\tr}|^2)^\frac{q}{2}+C_1,
\end{align*}
where $C_1:=(1+|\xi(0)|^2)^{\frac{q}{2}}+\frac{q\tau}{2} K_1(1+\|\xi\|^2)^\frac{q}{2}+
\frac{q\tau}{2} K_2\max_{-N\leq j\leq0}V_1(\xi(t_j))$.
Applying the discrete Gronwall inequatity and the fact $i\tr\leq T$ yield
\begin{align*}
&\E\left[\left(1+|z_i^{\tr}|^2\right)^\frac{q}{2}\right]\leq C_1e^{Ci\tr }\leq C_1e^{CT}\leq C.
\end{align*}
\end{proof}
\subsection{The strong convergence}\la{SR}
This section concerns the strong convergence of the TEM scheme. We begin with a probability estimation.
\begin{lemma}\la{L:C_1}
Assume that  $(\textup{H}1)$-$(\textup{H}3)$ hold.  For any    $\tr,~ \tr_1\in(0, 1]$, let
\begin{align}\la{3.18}
\varrho_{\tr_1}^\tr:=\inf\{t\geq 0:|\breve{z}_{\tr}(t)|\geq  \Phi^{-1}(h_{\Phi,\mu}(\tr_1))\}.
\end{align}
Then  for any   $T>0$   and  $\tr\in(0,\tr_1]\subseteq (0, 1]$,
\begin{align}\la{s4.18}
 \PP {\{\varrho_{\tr_1}^\tr \leq T\}}\leq \frac {C}{ \big( {\Phi}^{-1}(h_{\Phi,\mu}(\tr_1))\big)^q}.
\end{align}
  \end{lemma}
\begin{proof}
\textbf{Proof.} Set $\zeta_{\tr_1}^\tr:=\inf\{i\geq 0:|\breve{z}_{i}^\tr |\geq  \Phi^{-1}(h(\tr_1))\}$.
Define
$$\breve{f}_i=f(\breve{z}_{i}^\tr , \breve{z}_{i-N}^\tr ),~~\breve{g}_i=g(\breve{z}_{i}^\tr , \breve{z}_{i-N}^\tr ),$$
 and
 $$\breve f_{i\wedge \zeta_{\tr_1}^\tr}:=f(\breve z_{i\wedge\zeta_{\tr_1}^\tr}^\tr, \breve z_{(i-N)\wedge\zeta_{\tr_1}^\tr}^\tr),
  ~~\breve g_{i\wedge\zeta_{\tr_1}^\tr}:=g(\breve z_{i\wedge\zeta_{\tr_1}^\tr}^\tr, \breve z_{(i-N)\wedge\zeta_{\tr_1}^\tr}^\tr).$$
For any $i\geq 1$, if $\omega\in \{\zeta_{\tr_1}^\tr\geq i\}$,
 it is obvious that $z_{i-1}^\tr=\breve z_{i-1}^\tr$, $z_{i-1-N}^{\tr}=\breve z_{i-1-N}^{\tr}$ and
\begin{align*}
\breve z_{i\wedge\zeta_{\tr_1}^\tr}^\tr=\breve z_{i}^\tr=&\breve z_{i-1}^\tr+\breve f_{i-1}\tr+\breve  g_{i-1}\tr W_{i-1}.
\end{align*}
Otherwise, $\omega\in \{\zeta_{\tr_1}^\tr< i\}$, we can then write
\begin{align*}
\breve z_{i\wedge\zeta_{\tr_1}^\tr}^\tr=\breve z_{\zeta_{\tr_1}^\tr}^\tr=\breve z_{(i-1)\wedge\zeta_{\tr_1}^\tr}^\tr.
\end{align*}
Combining both cases we have
\begin{align*} 
 \br z_{i\wedge\zeta_{\tr_1}^\tr}^\tr
  =\br z_{(i-1)\wedge \zeta_{\tr_1}^\tr}^\tr+\left[\br{f}_{(i-1)\wedge\zeta_{\tr_1}^\tr}\tr
  +\br{g}_{(i-1)\wedge\zeta_{\tr_1}^\tr} \tr W_{i-1}\right]\textbf 1_{[[0, \zeta_{\tr_1}^\tr]]}(i).
\end{align*}
Then,
\begin{align*}
 \left(1+|\br z_{i\wedge\zeta_{\tr_1}^\tr}^\tr|^2 \right)^{\frac{q}{2}}
  = \left(1+|\br z_{(i-1)\wedge\zeta_{\tr_1}^\tr}^\tr|^2\right)^{\frac{q}{2}}
  \left(1+\br{\Gamma}_{(i-1)\wedge\zeta_{\tr_1}^\tr }  \textbf 1_{[[0, \zeta_{\tr_1}^\tr]]}(i)\right)^{\frac{q}{2}},
\end{align*}
where
\begin{align*}
\br{\Gamma}_{(i-1)\wedge\zeta_{\tr_1}^\tr}
=&\big(1+|\br z_{(i-1)\wedge\zeta_{\tr_1}^\tr}^\tr|^2\big)^{-1}
\big[|\br{f}_{(i-1)\wedge\zeta_{\tr_1}^\tr}|^2\tr^2
+|\br{g}_{(i-1)\wedge\zeta_{\tr_1}^\tr} \tr W_{i-1}|^2\nn\\
&+2\langle \br z_{(i-1)\wedge\zeta_{\tr_1}^\tr}^\tr, \br f_{(i-1)\wedge\zeta_{\tr_1}^\tr}\rangle \tr+2\langle \br z_{(i-1)\wedge\zeta_{\tr_1}^\tr}^\tr,
\br{g}_{(i-1)\wedge\zeta_{\tr_1}^\tr}\tr W_{i-1}\rangle\nn\\
&+2\langle\br{f}_{(i-1)\wedge\zeta_{\tr_1}^\tr},\br{g}_{(i-1)\wedge\zeta_{\tr_1}^\tr} \tr W_{i-1}\rangle\tr\big].
\end{align*}
Similar, we obtain that $1\leq i\leq [T/\tr]$
\begin{align}\la{ss4.24}
&  \E  \left[(1+|\br z_{i\wedge\zeta_{\tr_1}^\tr}^\tr|^2)^{\frac{q}{2}}\big|
\mathcal F_{t_{(i-1)\wedge\zeta_{\tr_1}^\tr}}\right] \nn\\
   \leq & (1 + |\br z_{(i-1)\wedge\zeta_{\tr_1}^\tr}^\tr|^2)^{\frac{q}{2}}
   \big[1 + \frac{q}{2} \E(\br{\Gamma}_{(i-1)\wedge\zeta_{\tr_1}^\tr}
   \textbf 1_{[[0, \zeta_{\tr_1}^\tr]]}(i)|\mathcal{F}_{t_{i-1}\wedge\zeta_{\tr_1}^\tr})\nn\\
&     + \frac{q(q-2)}{8} \E(\br{\Gamma}_{(i-1)\wedge\zeta_{\tr_1}^\tr}^2
     \textbf 1_{[[0, \zeta_{\tr_1}^\tr]]}(i)|\mathcal F_{t_{i-1}\wedge\zeta_{\tr_1}^\tr})\nn\\
  &+ \E(\br{\Gamma}_{(i-1)\wedge\zeta_{\tr_1}^\tr}^3P_{k}(\br{\Gamma}_{(i-1)\wedge\zeta_{\tr_1}^\tr})
  \textbf 1_{[[0, \zeta_{\tr_1}^\tr]]}(i)|\mathcal{F}_{t_{i-1}\wedge\zeta_{\tr_1}^\tr})\big].
\end{align}
By virtue of the  martingale property of  $W(t)$ and the Doob martingale stopping
 time theorem \cite[p.11, Theorem 3.3]{Mao2007},  we have
\begin{align}\label{ss4.21}
&\E\left((A \tr W_i)\textbf 1_{[[0,\zeta_{\tr_1}^\tr]]}(i)\big|
\mathcal{F}_{t_{(i-1)\wedge\zeta_{\tr_1}^\tr}}\right)=0,\nn\\
&\E\left(|A\tr W_i|^{2}\textbf 1_{[[0,\zeta_{\tr_1}^\tr]]}(i)\big|
\mathcal{F}_{t_{(i-1)\wedge\zeta_{\tr_1}^\tr}}\right)
=|A|^2\tr\E\left(\textbf 1_{[[0,\zeta_{\tr_1}^\tr]]}(i)\big|
\mathcal{F}_{t_{(i-1)\wedge\zeta_{\tr_1}^\tr}}\right),
\end{align}
and
\begin{align}\la{ssg4.21}
&\E\left((A \tr W_i)^{2j-1}\textbf 1_{[[0,\zeta_{\tr_1}^\tr]]}(i)\big|
\mathcal{F}_{t_{(i-1)\wedge\zeta_{\tr_1}^\tr}}\right)=0,\nn\\
&\E\left(|A\tr W_i|^{j}\textbf 1_{[[0,\zeta_{\tr_1}^\tr]]}(i)\big|
\mathcal{F}_{t_{(i-1)\wedge\zeta_{\tr_1}^\tr}}\right)
\leq C\tr^{\frac{j}{2}}\E\left(\textbf 1_{[[0,\zeta_{\tr_1}^\tr]]}(i)\big|
\mathcal{F}_{t_{(i-1)\wedge\zeta_{\tr_1}^\tr}}\right),
\end{align}
where $A\in\mathbb{R}^{d\times m}, ~j\geq 2$.
Using these and (H2), by the same way as Theorem \ref{th3} we yield
\begin{align}\la{s4.28}
&\E\Big[ (1+|\br z_{i\wedge\zeta_{\tr_1}^\tr}^\tr|^2)^{\frac{q}{2}}\big|
\mathcal{F}_{t_{(i-1)\wedge\zeta_{\tr_1}^\tr}}\Big]\nn\\
 \leq&  (1
 +C\tr)
 \left(1+|\br z_{i\wedge\zeta_{\tr_1}^\tr}^\tr|^2\right)^{\frac{q}{2}} +\frac{q}{2}K_1\tr \big(1+ |\br z_{(i-1-N)\wedge\zeta_{\tr_1}^\tr}^{\tr}|^2\big)^{\frac{q}{2}}\nn\\
& -\frac{q}{2}K_2\tr  V_1(\br z_{(i-1)\wedge\zeta_{\tr_1}^\tr}^\tr) \E\left(\textbf 1 _{[[0,\zeta_{\tr_1}^\tr]]}(i)|\mathcal{F}_{t_{(i-1)\wedge\zeta_{\tr_1}^\tr}}\right)\nn\\
& +\frac{q}{2}K_2\tr V_1(\br z_{(i-1-N)\wedge\zeta_{\tr_1}^\tr}^{\tr}) \E\left(\textbf 1 _{[[0, \zeta_{\tr_1}^\tr]]}(i)|\mathcal{F}_{t_{(i-1)\wedge\zeta_{\tr_1}^\tr}}\right),
\end{align}
 which implies
\begin{align*}
&\E\left[\left(1+|\br z_{i\wedge\zeta_{\tr_1}^\tr}^\tr|^2\right)^\frac{q}{2}\right]\nn\\
\leq&(1+|\xi(0)|^2)^{\frac{q}{2}}+C\tr \sum_{k=0}^{i-1}
\E\left(1+|\br z_{k\wedge\zeta_{\tr_1}^\tr}^{\tr}|^2\right)^\frac{q}{2}\nn\\
&+\frac{Nq}{2}K_1\tr \left(1+\|\xi\|^2\right)^\frac{q}{2}+\frac{q}{2}K_1\tr \sum_{k=0}^{(i-1-N)\vee 0}
\E\left(1+|\br z_{k\wedge\zeta_{\tr_1}^\tr}^{\tr}|^2\right)^\frac{q}{2}\nn\\
&-\frac{q}{2}K_2\tr \sum_{k=0}^{i-1}
\E\bigg[V_1(\br z_{k\wedge\zeta_{\tr_1}^\tr}^{\tr})\E\left(\textbf 1_{[[0,\zeta_{\tr_1}^\tr]]}(k+1)|
\mathcal{F}_{t_{k\wedge\zeta_{\tr_1}^\tr}}\right)\bigg]\nn\\
&+\frac{Nq}{2}K_2\tr \max_{-N\leq j\leq0}V_1(\xi(t_j))\nn\\
&+\frac{q}{2}K_2\tr \sum_{k=0}^{(i-1-N)\vee 0}
\E\bigg[V_1(\br z_{k \wedge\zeta_{\tr_1}^\tr}^{\tr})
\E\left(\textbf 1_{[[0,\zeta_{\tr_1}^\tr]]}(k+1+N)|\mathcal{F}_{t_{(k+N)\wedge\zeta_{\tr_1}^\tr}}\right)\bigg].
\end{align*}
Due to the fact $\textbf 1_{[[0,\zeta_{\tr_1}^\tr]]}(k+1+N)\leq \textbf 1_{[[0,\zeta_{\tr_1}^\tr]]}(k+1)$,
one observes
\begin{align}\la{ss4.27}
\E\left[\left(1+|\br z_{i\wedge\zeta_{\tr_1}^\tr}^\tr|^2\right)^\frac{q}{2}\right]
\leq&C\tr \sum_{k=0}^{i-1}
\E\left(1+|\br z_{k\wedge\zeta_{\tr_1}^\tr}^{\tr}|^2\right)^\frac{q}{2}+C_1,
\end{align}
 where $C_1:=(1+|\xi(0)|^2)^{\frac{q}{2}}+\frac{q\tau}{2} K_1\left(1+\|\xi\|^2\right)^\frac{q}{2}+\frac{q\tau}{2} K_2\max_{-N\leq j\leq0}V_1(\xi(t_j))$.
Applying the discrete Gronwall inequality  together with $i\tr\leq T$ implies
\begin{align*}
\E\left[\left(1+|\br z_{i\wedge\zeta_{\tr_1}^\tr}^\tr|^2\right)^\frac{q}{2}\right]
\leq C_1e^{Ci\tr }\leq C_1e^{CT}.
\end{align*}
Therefore the required assertion follows from
\begin{align*}
    \big( \Phi^{-1}(h_{\Phi,\mu}(\tr_1))\big)^q  \PP {\{\varrho_{\tr_1}^\tr \leq T\}}
    \leq\E\big[|\br z(T\wedge \varrho_{\tr_1}^\tr) |^q\big]
    \leq& \E\left(1+|\br z_{ [\frac{T}{\tr}]\wedge\zeta_{\tr_1}^\tr}^{\tr}|^2\right)^\frac{q}{2}\leq C.
\end{align*}
\end{proof}

Now we establish the $q$th moment convergence of the TEM scheme \eqref{s3.7} for $q>0$.
\begin{theorem}\la{th4}
Assume that $(\textup{H}1)$-$(\textup{H}3)$ hold. Then for any $p\in (0, q)$,
\be \la{s3.10}
\lim_{\tr\rightarrow 0^+} \E |x(T)-z_{\tr}(T)|^p=0,~~~~\forall ~T> 0.
\ee
\end{theorem}
\begin{proof}\textbf{Proof.}
For any $M>\Phi^{-1}(h_{\Phi, \mu}(1))$, choose $\tr_{1}\in(0,1)$ such that $\Phi^{-1}(h_{\Phi, \mu}(\tr_1))=M$.
Define
$
\theta_M^{\tr_1}=\vartheta_M \wedge \varrho_{\tr_1}^\tr,
$
where $\vartheta_M$ and $\varrho_{\tr_1}^\tr$ are defined by (\ref{sM_1}) and (\ref{3.18}), respectively.
For any $\kappa_1>0$, by Young's inequality
\begin{align}\la{s4.31}
 \E|x(T)-z_{\tr}(T)|^p
 =& \E\lf(|x(T)-z_{\tr}(T)|^p \textbf 1_{\{\theta_M^{\tr_1} > T\}}\rt)
+ \E\lf(|x(T)-z_{\tr}(T)|^p \textbf 1_{\{\theta_M^{\tr_1} \leq T\}}\rt) \nonumber \\
\leq & \E\lf(|x(T)-z_{\tr}(T)|^p \textbf 1_{\{\theta_M^{\tr_1} > T\}}\rt)+ \frac{p \kappa_1}{q}\E\lf(|x(T)-z_{\tr}(T)|^{q} \rt)\nn\\
&+ \frac{q-p}{q \kappa_1^{p/(q-p)}} \PP {\{\theta_M^{\tr_1} \leq T\}}.
\end{align}
It follows from  Theorem  \ref{th1} and Theorem \ref{th3} that
\begin{align*}
\frac{p \kappa_1}{q}\E\lf(|x(T)-z_{\tr}(T)|^{q}  \rt)\leq
\frac{p \kappa_1}{q}2^{q} (\E |x(T)|^{q}
+ \E |z_{\tr}(T)|^{q})
\leq C\frac{p \kappa_1}{q}.
\end{align*}
For any $\varepsilon_1>0$, choose $\kappa_1(\varepsilon_1)>0 $ small sufficiently
such that $C {p \kappa_1}/{q}\leq {\varepsilon_1}/{3}$.
Then
\begin{align}\label{ss4.32}
\frac{p \kappa_1}{q}\E\lf(|x(T)-z_{\tr}(T)|^{q}  \rt)
\leq\frac{\varepsilon_1}{3}.
\end{align}
Then we go a further step to  choose $M>\|\xi\|\vee\Phi^{-1}(h_{\Phi, \mu}(1))$  such that
$ {C(q-p)}/({M^{q} {q} \kappa_1^{p/({q}-p)}})\leq    {\varepsilon_1}/{6}$ and choose $\tr_1\in (0, 1]$ such that
$
 \Phi^{-1}(h_{\Phi, \mu}(\tr_1))= M.
$
From  (\ref{s2.6}) and (\ref{s4.18}) we obtain that
\begin{align}\la{s4.33}
&\frac{ q-p }{   {q} \kappa_1^{p/(q-p)}} \PP {\{\theta_M^{\tr_1} \leq T\}}\nn\\
\leq &   \frac{ q-p }{   {q} \kappa_1^{p/({q}-p)}}
\left(   \PP {\{\vartheta_{M } \leq T\}}+ \PP {\{\varrho_{  \tr}^{\tr_1} \leq T\}}\right)\nn\\
\leq &\frac{ q-p }{   {q} \kappa_1^{p/({q}-p)}}
\left(\frac{C}{M^q}+\frac{C}{[\Phi^{-1}(h_{\Phi, \mu}(\tr_1))]^q}\right)
=\frac{2 C({q}-p)}{ M^{q} {q} \kappa_1^{p/({q}-p)}} \leq \frac{\varepsilon_1}{3}  .
\end{align}
So it is sufficient for \eqref{s3.10} to show
$$\lim_{\tr\rightarrow 0^+}\E\lf(|x(T)-z_{\tr}(T)|^p \textbf 1_{\{\theta_M^{\tr_1} > T\}}\rt)=0.$$
For this purpose, we define
\begin{align*}
&f_M(x, y)=f\left(\lf(|x|\wedge M \rt) \frac{x}{|x|},\lf(|y|\wedge M \rt) \frac{y}{|y|}\right),\nn\\
&g_M(x, y)=g\left(\lf(|x|\wedge M \rt) \frac{x}{|x|}, \lf(|y|\wedge M \rt) \frac{y}{|y|}\right).
\end{align*}
Then  (H1) implies that for any $x,~\bar x,~y\in R^d$,
  \begin{align}\label{S4.3}
  |f_M(x, y)-f_M(\bar x, y)|\vee  |g_M(x, y)-g_M(\bar x, y)|\leq L_M|x-\bar x|.
  \end{align}
  Clearly, by  (\ref{S4.3}) and (H3), we have
 \begin{align}\label{S4.4}
 |f_M(x, y)|\vee  |g_M(x, y)|\leq&( L_M\vee \max_{|y|\leq M}|f(\textbf{0},y)|\vee\max_{|y|\leq M}|g(\textbf{0},y)|)(1+|x|).
 \end{align}
So we consider the linear SDDE
\be\la{F_5}
du(t) =f_M(u(t), u(t-\tau))dt +g_M(u(t), u(t-\tau))dW(t),
\ee
with  the initial data $\xi\in \mathcal{C}([-\tau, ~0]; ~\RR^d)$. Due to \cite[Theorem 2.1]{Sabanis2013}
  SDDE (\ref{F_5}) has a unique global solution $u(t)$ on $t\geq -\tau$.  Let $Y_{\tr}(t)$ be the piecewise
 EM solution of (\ref{F_5}). By (H3), (\ref{S4.3}) and (\ref{S4.4}) and according to \cite[Theorem 1]{Kumar-Sabanis}, it has the property
\begin{align}\label{s4.38}
\lim_{\tr\rightarrow 0^+}\E\Big[ \sup_{0\leq t\leq T} |u(t)-Y_{\tr}(t)|^{\tilde{p}}\Big]=0,
~\forall ~T> 0,~\tilde{p}>0.
\end{align}
Obviously,
\be\la{F_6}
x(t\wedge \vartheta_M)=u(t\wedge \vartheta_M),~\forall~ t\geq 0,~~\hbox{a.s}.
\ee
For any $\tr\in(0,\tr_1]$, the fact $\Phi^{-1}(h_{\Phi,\mu}(\tr))\geq\Phi^{-1}(h_{\Phi,\mu}(\tr_1))=M$ implies
\begin{align}\la{F_8}
z_{\tr}(t\wedge\theta_M^{\tr_1})=\br z_{\tr}(t\wedge\theta_M^{\tr_1})=Y_{\tr}(t\wedge\theta_M^{\tr_1}),.
~~\forall ~t~\geq0,~\hbox{a.s}.
\end{align}
Combining  (\ref{s4.38})-(\ref{F_8}) derives
\begin{align}\label{sss4.36}
&\lim_{\tr\rightarrow 0^+}\E\lf(x(T)-z_{\tr}(T)|^p \textbf 1_{\{\theta_M^{\tr_1} > T\}}\rt)\nn\\
 \leq&\lim_{\tr\rightarrow 0^+}\E\lf(|x(T\wedge\theta_M^{\tr_1})-z_{\tr}(T\wedge\theta_M^{\tr_1})|^p\rt)\nn\\
=&\lim_{\tr\rightarrow 0^+}\E\lf(|u(T\wedge\theta_M^{\tr_1})-Y_{\tr}(T\wedge\theta_M^{\tr_1})|^p\rt)\nn\\
\leq &\lim_{\tr\rightarrow 0^+}\E\left(\sup_{0\leq t\leq T}|u(t\wedge\theta_M^{\tr_1})-Y_{\tr}(t\wedge\theta_M^{\tr_1})|^p\right)\nn\\
\leq&\lim_{\tr\rightarrow 0^+}\E\left(\sup_{0\leq t\leq T}|u(t)-Y_{\tr}(t)|^p\right)=0.
\end{align}
Hence the proof is completed.
\end{proof}
\subsection{Convergence rate}\la{CR}
Furthermore, we shall obtain  the $\frac{1}{2}$ order  convergent rate of the TEM scheme $z_{\tr}(t)$ defined in (\ref{s3.7}).
We  first state below the relevant assumptions.

\textbf{(H4)} Assume that the initial data $\xi(t)$ satisfies the H\"{o}lder continuous with the index $\lambda\geq\frac{1}{2}$, i.e., for any $s_1,~s_2\in[-\tau, ~0]$, there exists a positive constant $K_3$ such that
\begin{align}\label{S3.1}
|\xi(s_1)-\xi(s_2)|\leq K_3|s_1-s_2|^\lambda.
\end{align}

\textbf{(H5)}
Assume that there is a pair of positive constants $\alpha$, $K_4$ such that for any $x,~\bar x,~y,~\bar y\in \RR^d$,
\be\label{s3.12}
 |f(x, y)-f(\bar{x},\bar y)|\leq  K_4(|x-\bar{x}|+|y-\bar y|)(1+|x|^{\alpha}+|\bar{x}|^{\alpha}+|y|^{\alpha}+|\bar y|^{\alpha}),
\ee
\be\label{s3.13}
|g(x, y)-g(\bar{x},\bar y)|^2\leq  K_4(|x-\bar{x}|^2+|y-\bar y|^2)(1+|x|^{\alpha}+|\bar{x}|^{\alpha}+|y|^{\alpha}+|\bar y|^{\alpha}).
\ee

\textbf{(H6)}
 ~Assume that there exist  positive constants  $2\leq r\leq \frac{q}{\alpha+3}\wedge \frac{q}{2\alpha}$,  $\beta>r-1$,  $K_5$,  and  a function $\hat V(\cdot,\cdot)\in\mathcal{V}(\RR^d\times\RR^d;~\RR_+)$,  such that
\begin{align*}
&|x-\bar x|^{r-2}[2\langle x-\bar{x}, f(x, y)-f(\bar{x},\bar y)\rangle+\beta|g(x, y)-g(\bar{x}, \bar y)|^2]\nn\\
\leq&K_5(| x-\bar{y} |^{r}+|y-\bar y|^{r})-\hat V(x, \bar x)+\hat V(y, \bar y),~~~~\forall ~x,~\bar x,~y,~\bar y\in\RR^d.
\end{align*}


One notices from \eqref{s3.13} that
\begin{align}\label{cond-3}
|g(x, y)|\leq& |g(x, y)-g(\mathbf{0}, \mathbf{0})|+|g(\mathbf{0}, \mathbf{0})|\nn\\
\leq& \sqrt {K_4}(|x|+|y|)(1+|x|^{\frac{\alpha}{2}}+|y|^{\frac{\alpha}{2}})+ |g(\mathbf{0}, \mathbf{0})|\nn\\
\leq& C(1+|x|^{\frac{\alpha}{2}+1}+|y|^{\frac{\alpha}{2}+1}).
\end{align}
\begin{rem}\label{r2} {
Due to \eqref{s3.2} and $(\textup{H}5)$, we may take
   \begin{align*}
\Phi(l)=&[|f(\mathbf{0},\mathbf{0})|+3l^{\alpha+1}K_4]\vee 2[|g(\mathbf{0}, \mathbf{0})|^2+3l^{\alpha+2}K_4]
\leq&|f(\mathbf{0},\mathbf{0})|\vee2|g(\mathbf{0},\mathbf{0})|^2+6l^{\alpha+2}K_4,
\end{align*}
 where $l\geq1$. Then
 \begin{align}\label{s3.20}
\Phi^{-1}(l)=\left(\frac{l-|f(\mathbf{0}, \mathbf{0})|\vee2|g(\mathbf{0}, \mathbf{0})|^2}{6K_4}\right)^{\frac{1}{\alpha+2}},
\end{align}
 where $l\geq|f(\mathbf{0}, \mathbf{0})|\vee2|g(\mathbf{0}, \mathbf{0})|^2+6K_4$.
And let $\mu=\frac{r(\alpha+2)}{2(q-r)}\in(0,\frac{1}{2}]$. Thus \eqref{s3.4} implies
\begin{align}\la{ssh}
h_{\Phi,\mu}(\tr)=[|f(\mathbf{0},\mathbf{0})|\vee2|g(\mathbf{0},\mathbf{0})|^2+6(\|\xi\|\vee 1)^{\alpha+2}K_4]\tr^{-\frac{r(\alpha+2)}{2(q-r)}}.
\end{align}
}\end{rem}

In order to estimate  the   convergence rate of the  TEM scheme, we prepare a auxiliary  process $\tilde{z}_{\tr}(t)$ described by
\begin{align}\label{s4.41}
\left\{
\begin{array}{ll}
\tilde{z}_{\tr}(t)=z_{i}^\tr+ f(z_{i}^\tr, z_{i-N}^\tr)( t-t_i )+
 g(z_{i}^\tr, z_{i-N}^\tr)( W(t)-W(t_i) ), \forall ~t\in [t_i, t_{i+1}),&\\
 \tilde{z}_{\tr}(t)=\xi(t), \forall ~t\in[-\tau,~0].
\end{array}
\right.
\end{align}
Obviously, $  \tilde{z}_{\tr}(t_i) =z_{\tr}(t_i)=z_{i}^\tr$ for  $i\geq -N$.

\begin{lemma}\la{lemma+1}
  Assume that $(\textup{H}2)$ and $(\textup{H}5)$
  hold. Then  for any $\tilde r\in(0,2q/(\alpha+2)]$,       

  \be\la{s4.42}
   \sup_{0\leq t\leq T}  \E\big(|\tilde{z}_{\tr}(t)-z_{\tr}(t)|^{\tilde r} \big)
   \leq C \tr^{\frac{\tilde r}{2} },~~~~  \forall~ T> 0.
  \ee
  \end{lemma}
\begin{proof}\textbf {Proof.}
Fix $\tilde r\in(0,2q/(\alpha+2)]$.
Recalling  (\ref{s4.41}), we have that for any $t\in\big[t_i, t_{i+1}\big)$
 \begin{align*}
 &\E\big(|\tilde{z}_{\tr}(t)-z_{\tr}(t)|^{\tilde r} \big)=\E\big(|\tilde{z}_{\tr}(t)-z_{\tr}(t_i)|^{\tilde r} \big)\nn\\
 &\leq  2^{\tilde r} \E |  f(z_{i}^\tr, z_{i-N}^\tr) |^{\tilde r} \tr^{\tilde r}
  + 2^{\tilde r}\E\big( |  g(z_{i}^\tr, z_{i-N}^\tr)|^{\tilde r}  |W(t)-W(t_i)|^{\tilde r}\big) \\
 &\leq  C\lf( \E | f(z_{i}^\tr, z_{i-N}^\tr) |^{\tilde r}\tr^{\tilde r}
 +  \E |  g(z_{i}^\tr, z_{i-N}^\tr)|^{\tilde r} \tr^{\frac{\tilde r}{2}} \rt).
  \end{align*}
By  (\ref{s3.4}), (\ref{s3.5}), (\ref{cond-3}) and Theorem \ref{th3},
  \begin{align*}
 &\E\big(|\tilde{z}_{\tr}(t)-z_{\tr}(t)|^{{\tilde r} } \big)\nn\\
 \leq& C h_{\Phi,\mu}^{\tilde r}(\tr)\E(1+|z_{i}^\tr|)^{\tilde r}  \tr^{\tilde r}
 + C \E \big( 1+ |z_{i}^\tr|^{\frac{\alpha}{2}+1} +|z_{i-N}^\tr|^{\frac{\alpha}{2}+1}\big)^{\tilde r}
 \tr^{\frac{{\tilde r}}{2}}\nonumber\\
 \leq& C\big(1+(\E|z_{i}^\tr|^q)^{\frac{\tilde r}{q}}\big)\tr^{\frac{\tilde r}{2}}
 +C\big(1+(\E |z_{i}^\tr|^q)^{\frac{(\alpha+2)\tilde r}{{2q}}}+ (\E    |z_{i-N}^\tr|^q)^{\frac{(\alpha+2)\tilde r}{{2q}}}\big)\tr^{\frac{\tilde r}{2}}\nn\\
\leq & C \tr^{\frac{{\tilde r}}{2} },
 \end{align*}
which implies the required assertion.  \end{proof}

By the similar way as the Theorem \ref{th3} and Lemma \ref{L:C_1}, we yield the results for the auxiliary process.
\begin{lemma}\la{lemma+2}
Assume that $(\textup{H}1)$-$(\textup{H}3)$ hold.
 Then the auxiliary process  \eqref{s4.41} has the property
  \be\la{ss4.44}
     \sup_{0<\tr\leq 1}\sup_{0\leq t\leq T}\E|\tilde{z}_{\tr}(t)|^{q}\leq C,~~~~\forall ~T>0.
  \ee

  \end{lemma}

\begin{lemma}\la{lemma+3}
Assume that $(\textup{H}1)$-$(\textup{H}3)$ hold. For  any $\tr\in(0,1]$, let
 \be\la{cond-10}
\tilde{ \varrho}_{\tr} := \inf \{ t\geq -\tau: | {\tilde{z}}_{\tr}(t)|\geq  \Phi^{-1}(h_{\Phi,\mu}(\tr))\}.
 \ee
Then we have that  for any   $T>0$,
  \be\la{cond-11}
   \PP {\{\tilde{\varrho}_{\tr} \leq T\}}\leq \frac{K }{ \big( {\Phi}^{-1}(h_{\Phi,\mu}(\tr))\big)^{{q}} }.
  \ee

  \end{lemma}

We go  a  further step  to estimate the error between   the auxiliary process $\tilde{z}_{\tr}(t)$  and  the exact solution $x(t)$. Define $e(t)=x(t)-\tilde{z}_{\tr}(t)$ for  short, which satisfies
\begin{align*}
\mathrm{d}e(t)=&\int_{0}^{t}\big[f(x(s),x(s-\tau))-f(z_{\tr}(s),z_{\tr}(s-\tau))\big]\mathrm ds\nn\\
&+\int_{0}^{t}\big[g(x(s),x(s-\tau))-g(z_{\tr}(s),z_{\tr}(s-\tau))\big]\mathrm dW(s).
\end{align*}
\begin{lemma}\la{lemma+4}
Assume that $(\textup{H}2)$, $(\textup{H}4)$-$(\textup{H}6)$ hold.
Then one  has the property
\begin{align}\la{s4.46}
\E \big|e(T)\big|^{r}
\leq C\tr^{\frac{r}{2}},~~~~\forall ~T\geq0.
\end{align}

\end{lemma}
\begin{proof}Proof.
Define $\chi_{ \tr}=\vartheta_{ {\Phi}^{-1}(h_{\Phi,\mu}(\tr))}{ \wedge \varrho_{\tr}^\tr}
\wedge \tilde{\varrho}_{\tr}$,
where  $\vartheta_M$, $\varrho_{\tr}^\tr$ and $\tilde{\varrho}_{\tr}$ are
defined in (\ref{sM_1}), (\ref{3.18}) and (\ref{cond-10}),  respectively.
By Young's inequality
\begin{align}\la{cond-13}
\E|e(T)|^{r}
 =&\E\lf(|e(T)|^{r}
 \textbf 1_{\{\chi_{ \tr}>T\}}\rt)+\E\lf(|e(T)|^{r}
\textbf 1_{\{\chi_{ \tr}\leq T\}}\rt) \nonumber \\
 \leq &\E\lf(|e(T)|^{r}
  \textbf 1_{\{\chi_{ \tr}>T\}}\rt)+ \frac{r\tr^{\frac{r}{2}}}{{q}}\E|e(T)|^{q}
  + \frac{q-r}
{q\tr^{\frac{ r^2}{2(q-r)}}}\PP(\chi_{ \tr}\leq T).
\end{align}
By   Theorem \ref{th1} and Lemma \ref{lemma+2}
\begin{align}\label{cond-19}
 \frac{r\tr^{\frac{r}{2}}}{{q}}\E|e(T)|^{q}
  \leq  2^{q-1}\frac{r\tr^{\frac{r}{2}}}{{q}}\left(\E|x(T)|^q+\E|\tilde z_{\tr}(T)|^q\right)
  \leq C\tr^{\frac{r}{2}}.
\end{align}
Using   (\ref{s2.6}), (\ref{s4.18}), (\ref{cond-11}), and then by (\ref{s3.20}) and (\ref{ssh}), we have
\begin{align}\label{sss4.44}
 &\frac{q-r}{q\tr^{\frac{ r^2}{2(q-r)}}}
 \PP(\chi_{ \tr}\leq T)\nn\\
   \leq &   \frac{q-r}{q\tr^{\frac{r^2}
  {2(q-r)}}}  \Big( \PP {\{\vartheta_{\Phi^{-1}(h_{\Phi,\mu}(\tr)) } \leq T\}}
   +\PP {\{\varrho_{  \tr}^{\tr} \leq T\}} + \PP {\{\tilde{\varrho}_{  \tr} \leq T\}} \Big)\nn\\
 \leq &\frac{q-r}{q\tr^{\frac{r^2}{2(q-r)}}}
   \frac{3C}{( {\Phi}^{-1}(h_{\Phi,\mu}(\tr)))^{{q}}}
 \leq C\tr^{\frac{qr}{2(q-r)}-\frac{r^2}{2(q-r)}}=C\tr^{\frac{r}{2}}.
\end{align}
Next we estimate the first term on the right hand of (\ref{cond-13}).
Using the It\^{o} formula, we have
\begin{align}\label{ss27}
&\big| e(T\wedge\chi_\tr)\big|^{r}\nn\\
 \leq&\int_0^{T\wedge\chi_\tr}\frac{r}{2}|e(s)|^{r-2}
 \Big[2\langle e(s), f(x(s), x(s-\tau)) -f(z_{\tr}(s), z_{\tr}(s-\tau))\big\rangle\nn\\
 &+(r-1)\big|g(x(s), x(s-\tau))-g(z_{\tr}(s), z_{\tr}(s-\tau))\big|^2\Big]\mathrm{d}s\nn\\
 & + \int_0^{T\wedge\chi_\tr}r|e(s)|^{r-2}
  \langle e(s), g(x(s), x(s-\tau)) -g(z_{\tr}(s), z_{\tr}(s-\tau))\big\rangle\mathrm{d}W(s).
\end{align}
Due to $r\in[2,\beta+1)$, one chooses a constant $\kappa_2>0$,  such that $(1+\kappa_2)(r-1)\leq \beta$. It follows from the elementary inequality and (H5) that
\begin{align*}
&2\langle e(s), f(x(s), x(s-\tau))
-f(z_{\tr}(s), z_{\tr}(s-\tau))\big\rangle\nn\\
&+(r-1)\big|g(x(s), x(s-\tau))-g(z_{\tr}(s), z_{\tr}(s-\tau))\big|^2\nn\\
\leq&2\langle e(s), f(x(s), x(s-\tau))
-f(\tilde z_{\tr}(s), \tilde z_{\tr}(s-\tau))\big\rangle\nn\\
&+2\langle e(s), f(\tilde z_{\tr}(s), \tilde z_{\tr}(s-\tau))
-f( z_{\tr}(s), z_{\tr}(s-\tau))\big\rangle\nn\\
 &+(1\!+\kappa_2)(r-1)|g(x(s), x(s-\tau))\!
 -g(\tilde{z}_{\tr}(s), \tilde z_{\tr}(s-\tau))|^2 \nn\\
&+ \Big(1+\frac{1}{\kappa_2}\Big)(r-1)|
g(\tilde{z}_{\tr}(s), \tilde z_{\tr}(s-\tau))\!-g(z_{\tr}(s), z_{\tr}(s-\tau))|^2\nn\\
\leq&2\langle e(s), f(x(s), x(s-\tau))
-f(\tilde z_{\tr}(s), \tilde z_{\tr}(s-\tau))\big\rangle\nn\\
&+\beta|g(x(s), x(s-\tau))\!
 -g(\tilde{z}_{\tr}(s), \tilde z_{\tr}(s-\tau))|^2 \nn\\
 &+2K_4|e(s)|(|\tilde z_{\tr}(s)-z_{\tr}(s)|+|\tilde z_{\tr}(s-\tau)-z_{\tr}(s-\tau)|)
 (1+|\tilde z_{\tr}(s)|^{\alpha}\nn\\
 &+|\tilde z_{\tr}(s-\tau)|^{\alpha}+|z_{\tr}(s)|^{\alpha}+| z_{\tr}(s-\tau)|^{\alpha})\nn\\
 &+(1\!+\frac{1}{\kappa_2})(r-1)K_4(|\tilde z_{\tr}(s)-z_{\tr}(s)|^2+|\tilde z_{\tr}(s-\tau)-z_{\tr}(s-\tau)|^2)(1+|\tilde z_{\tr}(s)|^{\alpha}\nn\\
 &+|\tilde z_{\tr}(s-\tau)|^{\alpha}+|z_{\tr}(s)|^{\alpha}+| z_{\tr}(s-\tau)|^{\alpha}).
\end{align*}
Inserting the above inequality into \eqref{ss27} and using  (H6), we derive
\begin{align}\label{S4.5}
&(|e(T \wedge \chi_{ \tr}) |^{r} ) \nn\\
  \leq & \frac{r}{2}   \int_0^{T \wedge \chi_{ \tr} }
   \Big[ K_5|e(s)|^{r}+K_5|e(s-\tau)|^{r}
   -\hat V(x(s),\tilde z(s))\nn\\
   &+\hat V(x(s-\tau),\tilde z_{\tr}(s-\tau))
   +2K_4 |e(s)|^{r-1}|\tilde z_{\tr}(s)- z_{\tr}(s)|
   (1+|\tilde z_{\tr}(s)|^{\alpha}\nn\\
   &+|\tilde z_{\tr}(s-\tau)|^{\alpha}
   +|z_{\tr}(s)|^{\alpha}+|z_{\tr}(s-\tau)|^{\alpha})\nn\\
   &+2K_4 |e(s)|^{r-1}|\tilde z_{\tr}(s-\tau)- z_{\tr}(s-\tau)|
   (1+|\tilde z_{\tr}(s)|^{\alpha}+|\tilde z_{\tr}(s-\tau)|^{\alpha}
   +|z_{\tr}(s)|^{\alpha}\nn\\
   &+|z_{\tr}(s-\tau)|^{\alpha})+ K_4\Big(1+\frac{1}{\kappa_2}\Big)
   (r-1)|e(s)|^{r-2}|\tilde z_{\tr}(s) -z_{\tr}(s)|^{2}
   (1+|\tilde z_{\tr}(s)|^{\alpha}\nn\\
   &+|\tilde z_{\tr}(s-\tau)|^{\alpha}+|z_{\tr}(s)|^{\alpha}+|z_{\tr}(s-\tau)|^{\alpha})\nn\\
   &+ K_4\Big(1+\frac{1}{\kappa_2}\Big)
   (r-1)|e(s)|^{r-2}|\tilde z_{\tr}(s-\tau) -z_{\tr}(s-\tau)|^{2}
   (1+|\tilde z_{\tr}(s)|^{\alpha}\nn\\
   &+|\tilde z_{\tr}(s-\tau)|^{\alpha}+|z_{\tr}(s)|^{\alpha}+|z_{\tr}(s-\tau)|^{\alpha})\Big]\mathrm{d}s\nn\\
   & + \int_0^{T\wedge\chi_\tr}r|e(s)|^{r-2}
  \langle e(s), g(x(s), x(s-\tau)) -g(z_{\tr}(s), z_{\tr}(s-\tau))\big\rangle\mathrm{d}W(s).
\end{align}
Owing to $\hat V(x,x)=0$ for any $x\in\RR^d$,  we have
\begin{align}\label{S4.6}
& \int_0^{T \wedge \chi_{ \tr} }-\hat V(x(s),\tilde z_{\tr}(s))+\hat V(x(s-\tau),\tilde z_{\tr}(s-\tau))\mathrm{d}s\nn\\
\leq&\int_{-\tau}^{0}\hat V(x(s),\tilde z_{\tr}(s))\mathrm{d}s
=\int_{-\tau}^{0}\hat V(\xi(s),\xi(s))\mathrm{d}s
=0.
\end{align}
By (\ref{S4.5}), (\ref{S4.6}),  the Young inequality and the elementary inequality, we  yield
 \begin{align}\label{ss28}
&\E(|e(T \wedge \chi_{ \tr}) |^{r} ) \nn\\
\leq &\big (rK_5+2K_4(r-1)+K_4(1+\frac{1}{\kappa_2})(r-1)(r-2) \big)\E  \int_0^{T \wedge \chi_{ \tr} }
|e(s)|^{r}\mathrm{d}s\nn\\
&+K_4 \E  \int_0^{T}
|\tilde z_{\tr}(s) -z_{\tr}(s)|^{r}(1+|\tilde z_{\tr}(s)|^{\alpha}+|\tilde z_{\tr}(s-\tau)|^{\alpha}+|z_{\tr}(s)|^{\alpha}+|z_{\tr}(s-\tau)|^{\alpha})^{r}\mathrm{d}s\nn\\
&+K_4\E  \int_0^{T}
|\tilde z_{\tr}(s-\tau) -z_{\tr}(s-\tau)|^{r}(1+|\tilde z_{\tr}(s)|^{\alpha}+|\tilde z_{\tr}(s-\tau)|^{\alpha}+|z_{\tr}(s)|^{\alpha}\nn\\
&+|z_{\tr}(s-\tau)|^{\alpha})^{r}\mathrm{d}s\nn\\
&+K_4(1+\frac{1}{\kappa_2})(r-1)
 \E  \int_0^{T}|\tilde z_{\tr}(s) -z_{\tr}(s)|^{r}(1+|\tilde z_{\tr}(s)|^{\alpha}+|\tilde z_{\tr}(s-\tau)|^{\alpha}\nn\\
 &+|z_{\tr}(s)|^{\alpha}+|z_{\tr}(s-\tau)|^{\alpha})^\frac{r}{2}\mathrm{d}s\nn\\
 &+K_4(1+\frac{1}{\kappa_2})(r-1)
 \E  \int_0^{T}|\tilde z_{\tr}(s-\tau) -z_{\tr}(s-\tau)|^{r}(1+|\tilde z_{\tr}(s)|^{\alpha}+|\tilde z_{\tr}(s-\tau)|^{\alpha}\nn\\
 &+|z_{\tr}(s)|^{\alpha}+|z_{\tr}(s-\tau)|^{\alpha})^\frac{r}{2}\mathrm{d}s\nn\\
 \leq&C\E  \int_0^{T \wedge \chi_{ \tr} }|e(s)|^{r}\mathrm{d}s+J_1+J_2,
 \end{align}
 where
 \begin{align*}
 J_1:=&C\E  \int_0^{T}|\tilde z_{\tr}(s) -z_{\tr}(s)|^{r}(1+|\tilde z_{\tr}(s)|^{r\alpha}+|\tilde z_{\tr}(s-\tau)|^{r\alpha}+|z_{\tr}(s)|^{r\alpha}\nn\\
 &+|z_{\tr}(s-\tau)|^{r\alpha})\mathrm{d}s,\nn\\
 J_2:=&C\E  \int_0^{T}|\tilde z_{\tr}(s-\tau) -z_{\tr}(s-\tau)|^{r}(1+|\tilde z_{\tr}(s)|^{r\alpha}+|\tilde z_{\tr}(s-\tau)|^{r\alpha}+|z_{\tr}(s)|^{r\alpha}\nn\\
& +|z_{\tr}(s-\tau)|^{r\alpha})\mathrm{d}s.
\end{align*}
By H\"{o}lder's inequality, Theorem \ref{th3}, Lemma \ref{lemma+1},  and  Lemma \ref{lemma+2}, we have
\begin{align}\la{ss29}
J_1\leq& C \int_0^{T}(\E| \tilde z_{\tr}(s)- z_{\tr}(s)|^{2r})^{\frac{1}{2}}
(1+\E|\tilde z_{\tr}(s)|^{2r\alpha}+\E|\tilde z_{\tr}(s-\tau)|^{2r\alpha}\nn\\
 &+\E|z_{\tr}(s)|^{2r\alpha}+\E|z_{\tr}(s-\tau)|^{2r\alpha})^{\frac{1}{2}}\mathrm{d}s\nn\\
\leq&C\int_0^{T}\tr^{\frac{r}{2}}
\big[1+(\E|\tilde z_{\tr}(s)|^{q})^\frac{2r\alpha}{q}+(\E|z_{\tr}(s)|^{q})^\frac{2r\alpha}{q}+
(\E|\tilde z_{\tr}(s-\tau)|^{q})^\frac{2r\alpha}{q}\nn\\
&+(\E|z_{\tr}(s-\tau)|^{q})^\frac{2r\alpha}{q}\big]^\frac{1}{2} \mathrm{d}s\leq C\tr^{\frac{r}{2}}.
\end{align}
By the same way as $J_1$,  together with  (H4), we obtain
\begin{align}\label{S4.2}
J_2\leq&C\int_0^{T}(\E| \tilde z_{\tr}(s-\tau)- z_{\tr}(s-\tau)|^{2r})^{\frac{1}{2}}
\big[1+(\E|\tilde z_{\tr}(s)|^{q})^\frac{2r\alpha}{q}\nn\\
&+(\E|\tilde z_{\tr}(s-\tau)|^{q})^\frac{2r\alpha}{q}+(\E|z_{\tr}(s)|^{q})^\frac{2r\alpha}{q}+(\E|z_{\tr}(s-\tau)|^{q})^\frac{2r\alpha}{q}\big]^\frac{1}{2} \mathrm{d}s\nn\\
\leq&C\int_0^{T}(\E| \tilde z_{\tr}(s)- z_{\tr}(s)|^{2r})^{\frac{1}{2}}\mathrm{d}s+C\int_{-\tau}^{0}| \tilde z_{\tr}(s)- z_{\tr}(s)|^{r}\mathrm{d}s\nn\\
\leq&C\tr^{\frac{r}{2}}+C\int_{-\tau}^{0}|  \xi(s)- \xi([\frac{s}{\tr}]\tr)|^{r}\mathrm{d}s\nn\\
\leq&\tr^{\frac{r}{2}}+K_3^rC\int_{-\tau}^{0} |s-[\frac{s}{\tr}]\tr|^{\lambda r}\mathrm{d}s
\leq C\tr^{\frac{r}{2}}.
\end{align}
Inserting (\ref{ss29}), (\ref{S4.2}) into \eqref{ss28} and  applying the Gronwall inequality
 yields that
\begin{align}\la{ss17}
\E(|e(T \wedge \chi_{ \tau}) |^{r} )
\leq Ce^{CT}\tr^{r/2}.
\end{align}
Subsituting (\ref{cond-19}), (\ref{sss4.44}) and (\ref{ss17}) into (\ref{cond-13}), we  get the desired assertion.
\end{proof}

\begin{theorem}\label{th5}
 Assume that $(\textup{H}2)$, $(\textup{H}4)$-$(\textup{H}6)$ hold.
 Then  for any $\bar r\in(0,r]$, the TEM scheme $z_{\tr}(t)$ defined
 in \eqref{s3.7}
 has the property
\begin{align*}
\E \big|x(T)- z_{\tr}(T)\big|
^{\bar r}\leq C\tr^{\frac{\bar r}{2}}, ~~~~\forall~T>0.
\end{align*}
\end{theorem}
\begin{proof}\textbf{Proof.}
For any $T>0$, by (\ref{s4.42}) and (\ref{s4.46}), we  obtain
\begin{align*}
\E \big|x(T)- z_{\tr}(T)\big|^{r}
\leq2^{r}\E \big|x(T)- \tilde z_{\tr}(T)\big|^{r}
+2^{r}\E \big|\tilde z_{\tr}(T)- z_{\tr}(T)\big|^{r}
\leq C\tr^{\frac{r}{2}},
\end{align*}
which together with the  H\"{o}lder inequality implies the desired.
\end{proof}
\subsection{Exponential stability}\la{ES}
This section focuses on the exponential stability of SDDE \eqref{s2.2}. We firstly give the corresponding results on the exact solutions. Then we construct a more precise scheme to approximate the long-time behaviors of the system.
Without loss of generality,  we assume  $f(\mathbf{0},\mathbf{0})=0,~g(\mathbf{0},\mathbf{0})=0 $.
Moreover, 

 \textbf{(H7)} Assume that there exist  constants $\bar K_6>K_6>0$,~$ \bar K_7>K_7\geq0 $ and a function $V_2(\cdot)\in\mathcal{C}(\RR^d;   ~\RR_+)$ such that for any $ ~x,~y \in\RR^{d}$,
\begin{align}\label{s2.9}
\langle 2x, f(x,y)\rangle+|g(x,y)|^2
\leq -\bar K_6|x|^{2}+ K_6|y|^{2}-\bar K_7V_2(x)+ K_7V_2(y)
.
\end{align}

{
\textbf{(H8)} For any positive  constant $l_2$, there exist a positive constant $\hat L_{l_2}$ such that for any $|y|\leq l_2$
\begin{align*}
|f(\textbf 0,y)|+|g(\textbf 0,y)|\leq\hat L_{l_2}|y|.
\end{align*}}

  Using the  techniques of \cite[Theorem 3.4]{Mao2005} and \cite[Theorem 2.1]{li_mao2012}, we may get the exponential stability of SDDE \eqref{s2.2}.
\begin{theorem}\la{th2}
Assume that  $(\textup{H}1)$ and $(\textup{H}7)$ 
hold. Then 
 the solution $x(t)$ of   SDDE \eqref{s2.2}
 with an initial data $\xi\in{}\mathcal{C}([-\tau,~0]; ~\RR^d)$ has the property
  \begin{align*}
   \E|x(t)|^{2 } \leq Ce^{-\gamma t},~~~~\forall ~t>0,
   \end{align*}
   where $\gamma$ satisfies $ K_6e^{\gamma\tau}+\gamma\leq \bar K_6$ and
 $ K_7e^{\gamma\tau}\leq \bar K_7$.
  \end{theorem}
\begin{theorem}\la{th8}
Assume that  $(\textup{H}1)$ and $(\textup{H}7)$ hold. Then
 the solution $x(t)$ of   SDDE \eqref{s2.2}
 with an initial data $\xi\in{}\mathcal{C}([-\tau,~0]; ~\RR^d)$ has the property
  \begin{align*}
   \limsup_{t\rightarrow \infty}\frac{1}{t}\log|x(t)| \leq - \frac{\gamma}{2} ,~~~\hbox{a.s},
   \end{align*}
   where $ \gamma$ is defined in Theorem \textup{\ref{th2}}.
  \end{theorem}

Next we will give a more precise
numerical method keeping the underlying exponential stability  in  mean square and $\PP-1$.
~Under (H1) and (H8),  choose a strictly increasing continuous function $\hat\Phi : [1, \infty)\rightarrow \RR_+$ 
such that
\be\la{s3.23}
\sup_{|x|\vee |y|\leq l} \dis\left(\frac{ |f (x,y)|}{|x|+ 1\wedge|y|}\vee
\frac{| g(x, y)|^2 }{(|x|+1\wedge|y|)^2}\right)\leq \hat\Phi(l),~~~\forall~l\geq1.
\ee
For any given stepsize $\triangle\in(0,1]$, by \eqref{s3.4} we may take
\begin{align}\label{s3.26}
 h_{\hat\Phi,\mu}(\triangle )=\hat K\triangle^{- \mu},
\end{align}
 where $\hat K:=\hat\Phi(\|\xi\|\vee 1)$ and  $ \mu\in(0,\frac{1}{2})$.
Then the more precise TEM scheme is defined by
 \begin{align}\la{s3.28}
\left\{
\begin{array}{lll}
y_{i}^\tr=\xi(i\tr ),~\forall ~i=-N,\cdots, 0,&\\
\hat y_{i+1}^\tr=y_{i}^\tr+f( y_{i}^\tr, y_{i-N}^\tr)\tr +g( y_{i}^\tr,  y_{i-N}^\tr)\tr W_i,~\forall ~i=0,1,\cdots,& \\
  y_{i+1}^\tr=\Upsilon_{\hat\Phi,\mu}^\triangle (\hat y_{i+1}^\tr).&\\
\end{array}
\right.
\end{align}
So we have
\begin{align}\la{s3.27}
 \big|f(y_{i}^\tr, y_{i-N}^\tr)\big|
\leq& h_{\hat\Phi,\mu}(\tr)\big(|y_{i}^\tr|+1\wedge|y_{i-N}^\tr|\big), ~~~~~\forall ~x,~y\in\RR^d.
 \end{align}
 \begin{theorem}\la{th6}
Assume that  $(\textup{H}1)$, $(\textup{H}7)$ and $(\textup{H}8)$ hold. Then for any $\varepsilon\in(0, \gamma)$, there is $\bar \tr\in(0,1]$  such that 
for any $\tr \in(0, \bar\tr]$
  \be\la{Theorem+2e}
   \E| y_i^{\tr}|^{2 } \leq Ce^{(\gamma-\varepsilon) t_i},
   \ee
   where $\gamma$ is defined in Theorem \textup{\ref{th2}}.
  \end{theorem}
\begin{proof}\textbf{Proof.}
Define $f_i=f(y_{i}^\tr, y_{i-N}^\tr),~g_i=g(y_{i}^\tr, y_{i-N}^\tr),~\forall ~i\geq0$.
By (\ref{s3.28})
 \begin{align*}
&\E\Big[\big|\hat y_{i+1}^\tr\big|^2\big|\mathcal{F}_{t_i}\Big]\nn\\
   =&\E\Big[\big| y_{i}^\tr+f_i\tr+ g_i\tr W_i\big|^2\big|\mathcal{F}_{t_i}\Big]\nn\\
   =&\E\Big[\big(|y_{i}^\tr|^2+|f_i|^2\tr^2+ |g_i\tr W_i|^2
   +2\langle y_{i}^\tr,f_i\rangle\tr+2\langle y_{i}^\tr,g_i\tr W_i\rangle+2\langle f_i,g_i\tr W_i\rangle\tr\big)\big|\mathcal{F}_{t_i}\Big]\nn\\
   =&|y_{i}^\tr|^2+|f_i|^2\tr^2+ |g_i|^2\tr+2\langle y_{i}^\tr,f_i\rangle\tr,
\end{align*}
By (H7), \eqref{s3.26} and \eqref{s3.27}, we have
 \begin{align}\label{sss4.62}
&\E\Big[|\hat y_{i+1}^\tr\big|^2|\mathcal{F}_{t_i}\Big]\nn\\
\leq&|y_{i}^\tr|^2+\Big[-\bar K_6|y_{i}^\tr|^{2}+K_6|y_{i-N}^\tr|^{2}
-\bar K_7V_2(y_{i}^\tr)+K_7V_2(y_{i-N}^\tr)\Big]\tr\nn\\
&+  h_{\hat\Phi,\mu}^2(\tr)\big(|y_{i}^\tr|+1\wedge|y_{i-N}^\tr|\big)^2\tr^2\nn\\
\leq&[1-(\gamma\tr-2\hat K\tr^{2(1-\mu)})]|y_{i}^\tr|^2-K_6e^{\gamma\tau}\tr
|y_{i}^\tr|^2\nn\\
&+(K_6\tr+2\hat K\tr^{2(1-\mu)})|y_{i-N}^\tr|^{2}
-K_7e^{\gamma\tau}
\tr V_2(y_{i}^\tr)+K_7\tr V_2(y_{i-N}^\tr),
\end{align}
where constant $\gamma$ satisfies
 $$K_6e^{\gamma\tau}+\gamma\leq \bar K_6,~~~~~
 K_7e^{\gamma\tau}\leq \bar K_7.$$
Then taking the expectations on both sides of \eqref{sss4.62} we derive
\begin{align}\la{sso4_36}
&\E|y_{i+1}^\tr|^2-\E|y_{i}^\tr|^2\leq\E\big[\E(|\hat y_{i+1}^\tr|^2
|\mathcal{F}_{t_i})\big]-\E|y_{i}^\tr|^2\nn\\
\leq&-(\gamma\tr-2\hat K\tr^{2(1-\mu)})\E|y_{i}^\tr|^2-K_6e^{\gamma\tau}\tr\E|y_{i}^\tr|^2\nn\\
&+(K_6\tr+2\hat K\tr^{2(1-\mu)})\E|y_{i-N}^\tr|^{2}-K_7e^{\gamma\tau}\tr
\E V_2(y_{i}^\tr)+K_7\tr \E V_2(y_{i-N}^\tr),
\end{align}
which implies
\begin{align}\label{so2}
 \E|y_{i+1}^\tr|^2
&\leq |\xi(0)|^2-(\gamma\tr-2\hat K\tr^{2(1-\mu)})\sum_{k=0}^{i}\E|y_{k}^\tr|^2
-K_6e^{\gamma\tau}\tr\sum_{k=0}^{i}\E|y_{k}^\tr|^2\\
&+(K_6\tr
+2\hat K\tr^{2(1-\mu)})N\|\xi\|^2+(K_6\tr+2\hat K\tr^{2(1-\mu)})\sum_{k=0}^{(i-N)\vee 0}\E|y_{k}^\tr|^2\nn\\
&-K_7e^{\gamma\tau}\tr\sum_{k=0}^{i}\E V_2(y_{k}^\tr)+
K_7\tau\max_{-N\leq j\leq0}V_2(\xi(t_j))+K_7\tr\sum_{k=0}^{(i-N)\vee 0}\E V_2(y_{k}^\tr).
\end{align}
For any $\varepsilon\in(0,\gamma)$,  choose a constant $\bar \tr\in(0,~1]$
small sufficiently such that for any $\tr\in(0,\bar \tr]$
\begin{align}\label{so1}
&2\hat K\tr^{2(1-\mu)}\leq \varepsilon\tr,~~~~~K_6\tr+2\hat K\tr^{2(1-\mu)}\leq K_6 e^{\gamma\tau}\tr.
\end{align}
This together with (\ref{so2}) implies
\begin{align*}
\E|y_{i+1}^\tr|^2
\leq &-(\gamma-\varepsilon)\tr\sum_{k=0}^{i}\E|y_{k}^\tr|^2+C.
\end{align*}
A direct application of the discrete Gronwall inequatity  derives
\begin{align}\label{so3}
\E|y_{i+1}^\tr|^2
\leq Ce^{-(\gamma-\varepsilon)(i+1)\tr}=Ce^{-(\gamma-\varepsilon)t_{i+1}}.
\end{align}
Therefore  the desired result follows.
\end{proof}

Using the technique of \cite[Theorem 3.4]{wu_mao2013}, we yield the  almost sure exponential stability of the TEM scheme  (\ref{s3.28}).
 \begin{theorem}\la{th7}
Under the conditions of  Theorem {\rm\ref{th6}},  for any $\varepsilon\in(0, \gamma)$, there is $\bar \tr\in(0,1]$  such that 
for any $\tr \in(0, \bar\tr]$
  \be\la{Pas}
   \limsup_{i\rightarrow\infty}\frac{1}{i\tr}\log{| y_i^{\tr}|} \leq- \frac{\gamma-\varepsilon}{2}  ~~\hbox{a.s}.
   \ee
  \end{theorem}

\section{Numerical examples}\label{NE}
In this section to illustrate our results, we give two nonlinear  SDDE examples.
\begin{expl}\label{example1}
{\rm
Let us recall  SDDE \eqref{intr_exp1}
 ~and let $q=15$. By virtue of   \cite[p.211, Lemma 4.1]{Mao2007} we know
\begin{align*}
|a+b|^{p}\leq \frac{|a|^p}{\delta^{p-1}}+\frac{|b|^p}{(1-\delta)^{p-1}},
\end{align*}
where $a,~b\in \RR,~p>1,\delta\in(0,1)$.  This together with the  Young inequality implies
\begin{align*}
(1+|x|^{2})^{\frac{q}{2}-1}&\Big(\big\langle 2x,f(x,y)\big\rangle +14|g(x,y)|^2\Big)
\leq (2^{6.5}+\frac{14}{\delta^{5.5}}+\frac{14}{(1-\delta)^{5.5}})\nn\\
&+ 2^{7.5}|x|^{15}+\frac{14}{\delta^{5.5}}|y|^{15}-(16-\frac{182}{16(1-\delta)^{5.5}})|x|^{17}+ \frac{42}{16(1-\delta)^{5.5}}|y|^{17}.
\end{align*}
Choose $\delta$ small sufficiently such that $16-\frac{182}{16(1-\delta)^{5.5}}\geq\frac{42}{16(1-\delta)^{5.5}}$. So (H2) holds with $V_1(x)=(16-\frac{140}{16(1-\delta)^{5.5}})|x|^{17}$. By virtue of  Theorem \ref{th1}, (\ref{intr_exp1}) exists a unique global solution.

Let $r=3$ and $\beta>2$. By  the elementary inequality we have
\begin{align*}
2|x-\bar x|\langle x-\bar{x}, f(x, y)-f(\bar{x},\bar y)\rangle
\leq2|x-\bar x|^{3}-8|x-\bar x|^{3}(x^2+\bar x^2).
\end{align*}
Using Young's inequality and  the elementary inequality implies
\begin{align*}
 \beta|x-\bar x||g(x, y)-g(\bar{x}, y)|^2
\leq&\beta|x-\bar x||y-\bar y|^2(|y|^{\frac{1}{2}}+|\bar y|^{\frac{1}{2}})^2\nn\\
\leq&\frac{\beta^3}{3}|x-\bar x|^3+\frac{8}{3}|y-\bar y|^{3}(y^2+\bar y^2)+\frac{16}{3}|y-\bar y|^{3}.
\end{align*}
So we derive
\begin{align*}
&|x-\bar x|[2\langle x-\bar{x}, f(x, y)-f(\bar{x},\bar y)\rangle+\beta|g(x, y)-g(\bar{x}, \bar y)|^2]\nn\\
\leq&(2+\frac{\beta^3}{3})|x-\bar x|^3+\frac{16}{3}|y-\bar y|^{3}-8|x-\bar x|^{3}(x^2+\bar x^2)+8|y-\bar y|^{3}(y^2+\bar y^2).
\end{align*}
This implies (H6) holds with $\hat V(x,\bar x)=8|x-\bar x|^{3}(x^2+\bar x^2)$. Obviously, (H4) and (H5) hold with $\lambda=1, K_4=12$ and $\alpha=2$.
By (\ref{s3.20}), we take  $\Phi^{-1}(l)=(\frac{l}{72})^{\frac{1}{4}},~\forall ~l\geq72$.
By (\ref{ssh}) and $\|\xi\|=1$, then choose  $h(\tr)=72\tr^{-\frac{1}{2}},~\forall~\tr\in(0,1]$.
By  virtue of  Theorem \ref{th5}, the TEM scheme (\ref{s3.7}) satisfies that for any $\tr\in(0,1]$
\begin{align*}
\E \big|x(T)- z_{\tr}(T)\big|
^3\leq C\tr^{\frac{3}{2}},~~~~\forall ~T>0.
\end{align*}

We regard the numerical solution with small stepsize $\tr=2^{-20}$ as the exact solution $x(t)$, and carry out numerical experiments to compute  the error
$\E\left(|x(T)-z_{\tr}(T)|^{3}\right)$ between  the exact solution $x(T)$
and the numerical solution $z_{\tr}(T)$ of the TEM scheme using MATLAB.
 In Figure \ref{err_1}, the red solid line depicts   $\E\left(|x(T)-z_{\tr}(T)|^{3}\right)$ as the  function of $\tr$
for 1000 sample points  as $T=1,~2,~3$, and $\tr\in\{2^{-14},~2^{-12},~2^{-10},~2^{-8},~2^{-6}\}$.
The blue solid line plots  the reference function  $\tr^{\frac{3}{2}}$.
Figure \ref{err_1} supports the  result of  Theorem \ref{th5} that  the rate of $L^{3}$-convergence is $ {3}/{2}$.
 \begin{figure}[htp]
  \begin{center}
\includegraphics[width=14cm,height=8cm]{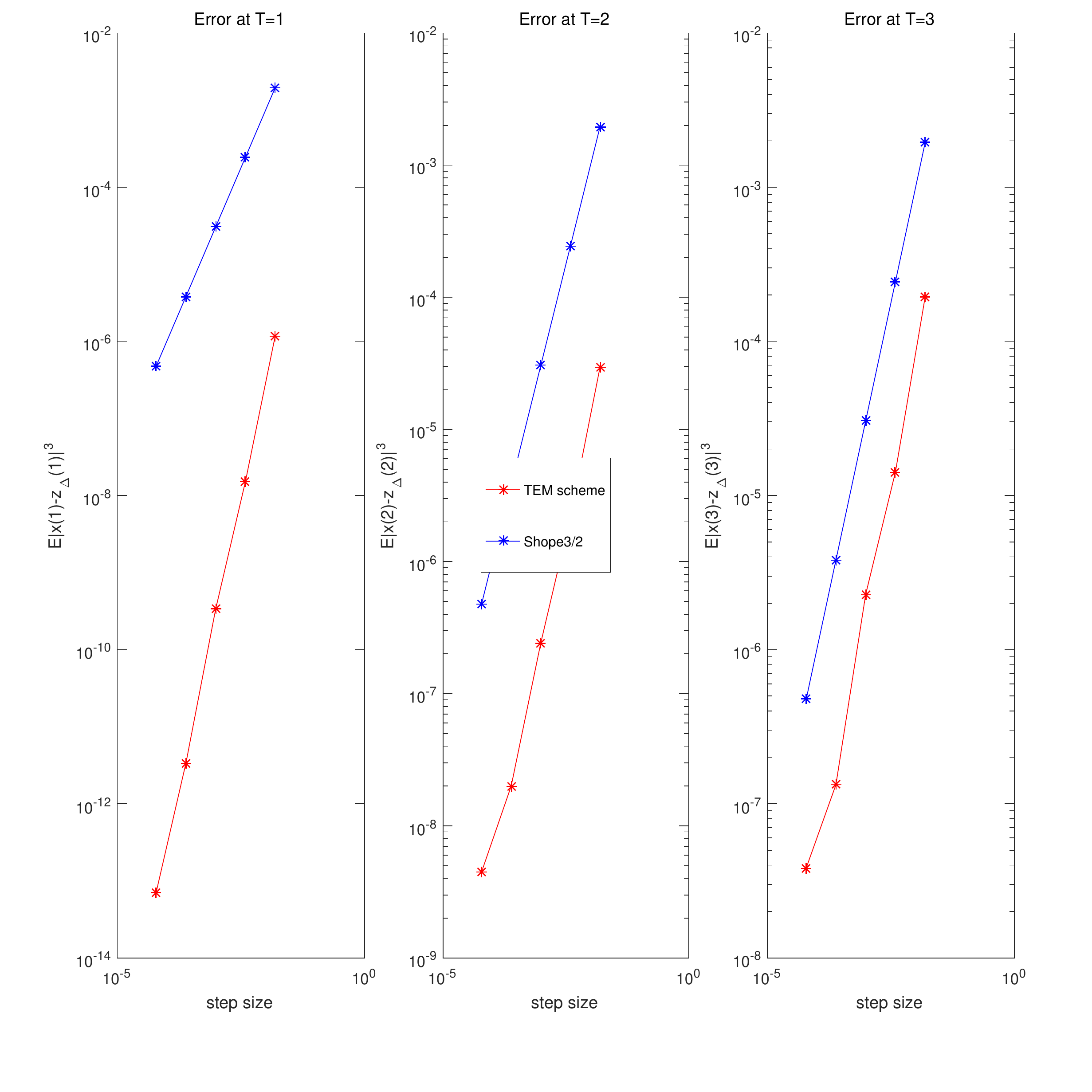}
   \end{center}
 \caption{The red solid line depicts the  $3$th moment approximation error $\E\left(|x(T)-z_{\tr}(T)|^{3}\right)$ between  the exact solution $x(T)$
and the TEM scheme $z_{\tr}(T)$, as the function of $\tr$
for 1000 sample points as $T=1,~2,~3$, and $\tr\in\{2^{-14},~2^{-12},~2^{-10},~2^{-8},~2^{-6}\}$.
The blue solid line plots  the reference  function  $\tr^{\frac{3}{2}}$.  }
\label{err_1}
\end{figure}
}
\end{expl}
\begin{expl}\label{eample2}
{\rm
Consider  2-dimensional SDDE
\begin{align}\la{ex2}
\left\{
\begin{array}{ll}
\mathrm{d}x_1(t)=\left(-\frac{3}{2}x_1(t)-x_1^{3}(t)\right)\mathrm{d}t+x_2^{2}(t-1)\mathrm{d}W_1(t),& \\
\mathrm{d}x_2(t)=\left(-x_2(t)-x_2^{3}(t)\right)\mathrm{d}t+x_1^{2}(t-1)\mathrm{d}W_2(t), ~~t> 0,
\end{array}
\right.
\end{align}
with  the initial data $(\xi_1(\theta),\xi_2(\theta))^T=(e^{-1.3\theta},0)^T, ~\theta\in[-1,~0]$.
We compute that for any $x,~y,~\bar x,~\bar y\in\RR^2$
\begin{align*}
&|f(x,y)-f(\textbf 0,y)|\leq\frac{3}{2} (1+|x|^2)|x|,~~~|g(x,y)-g(\textbf 0,y)|^2\leq 2|y|^4,\nn\\
&f(\textbf 0,y)=0,~~~|g(\textbf 0,y)|^2\leq 2|y|^4,
\end{align*}
and
\begin{align*}
&\langle 2x,f(x,y)\rangle+|g(x,y)|^2 \\
\leq& -2|x|^2-2(|x_1|^4+|x_2|^4)+(|y_1|^4+|y_2|^4).
\end{align*}
Then  (H1), (H7) and (H8) hold with $V_2(x)=|x_1|^4+|x_2|^4$, where $\bar K_6=2$,~$K_6=0.6$,~$\bar K_7=2$,~$K_7=1$.
 Choose $\gamma=0.69$ such that $$  1.8862\approx K_6e^{\gamma\tau}+\gamma<\bar K_6=2,
 ~~~~~1.9937\approx K_7e^{\gamma\tau}<\bar K_7=2.$$
By virtue of  Theorem \ref{th2} and  Theorem \ref{th8}, (\ref{ex2}) is the exponential stable in mean square and $\PP-1$.

By (\ref{s3.23}),  we take
$\hat\Phi(l)=2(1+l^2)$, where $l\geq1$. 
~By  (\ref{s3.26}) and $\|\xi\|=e^{1.3}$, we choose $h_{\hat\Phi,\mu}(\tr)=2(1+e^{2.6}) \tr^{-\frac{1}{100}},~\forall\tr\in(0,1]$.
Let $\varepsilon=0.5\in(0,\gamma)$. we choose $\bar\tr=2^{-7}$ such that for any $\tr\in(0,\bar \tr]$, (\ref{so1}) holds.
It follows from Theorem \ref{th6} and  Theorem \ref{th7} that for any $\tr\in(0, ~2^{-7}]$
$$
   \E|y_{i}^\tr|^{2 } \leq Ce^{-0.19 t_i},~  \forall~ i\geq 0,~~~~
    \limsup_{i\rightarrow\infty}\frac{1}{i\tr}\log{| y_{i}^\tr|} \leq -\frac{0.19}{2}, ~\hbox{a.s}.
  $$
 Figure  \ref{stabplot2} depicts the sample mean of the TEM scheme $y_i^{\tr}$ defined in \eqref{s3.28}.
Figure \ref{stabplot1} depicts a sample path of the EM solution $Y_{i}^{\tr}$  and  a sample path of the TEM solution $y_i^{\tr}$.
\begin{figure}[htp]
  \begin{center}
\includegraphics[width=14cm,height=7cm]{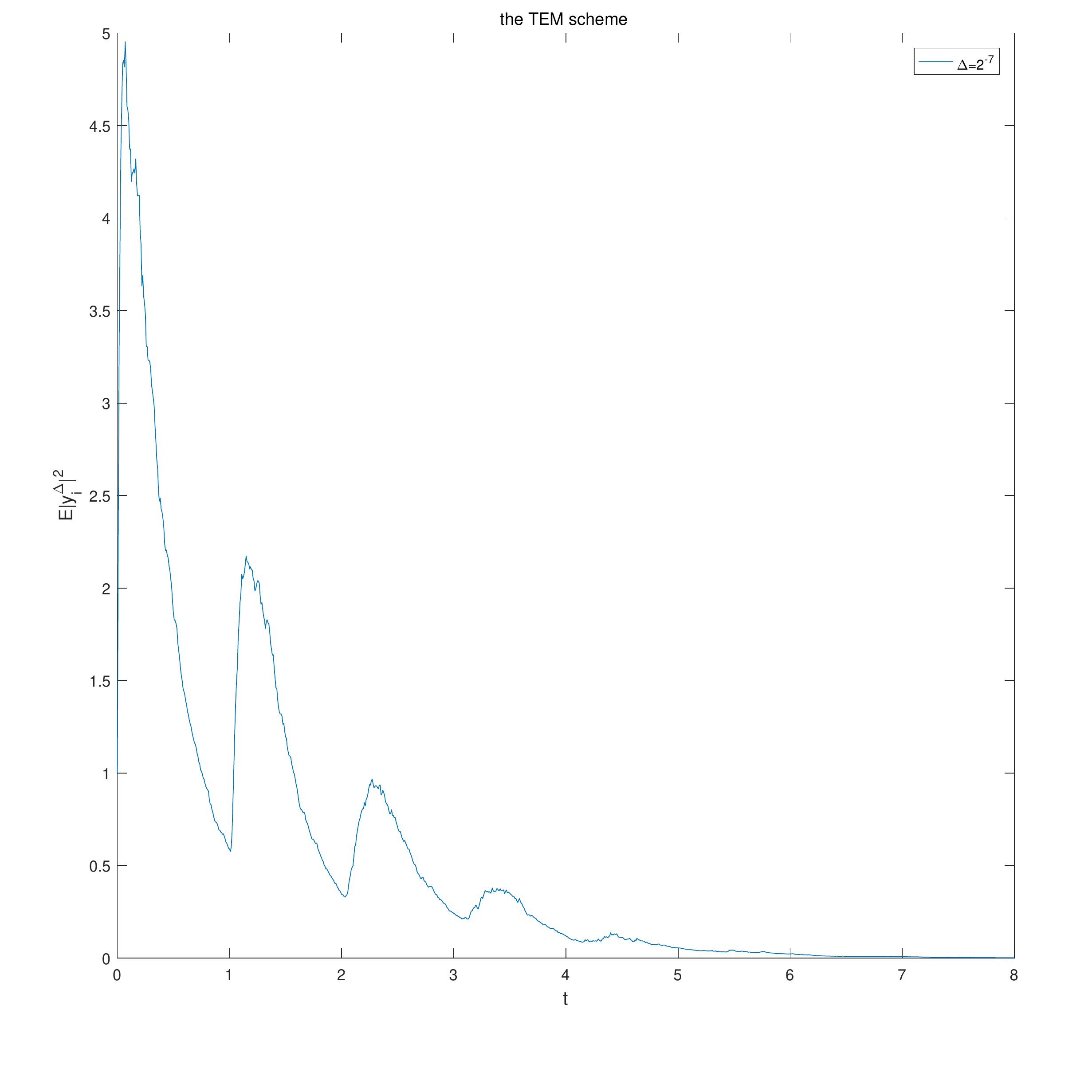}
   \end{center}
  \caption{The sample mean of $y_i^{\tr}$ with
  1000 sample points, 
  $\tr=2^{-7}$ and $t\in[0,~8]$.}
\label{stabplot2}
\end{figure}
\begin{figure}[htp]
  \begin{center}
\includegraphics[width=14cm,height=7cm]{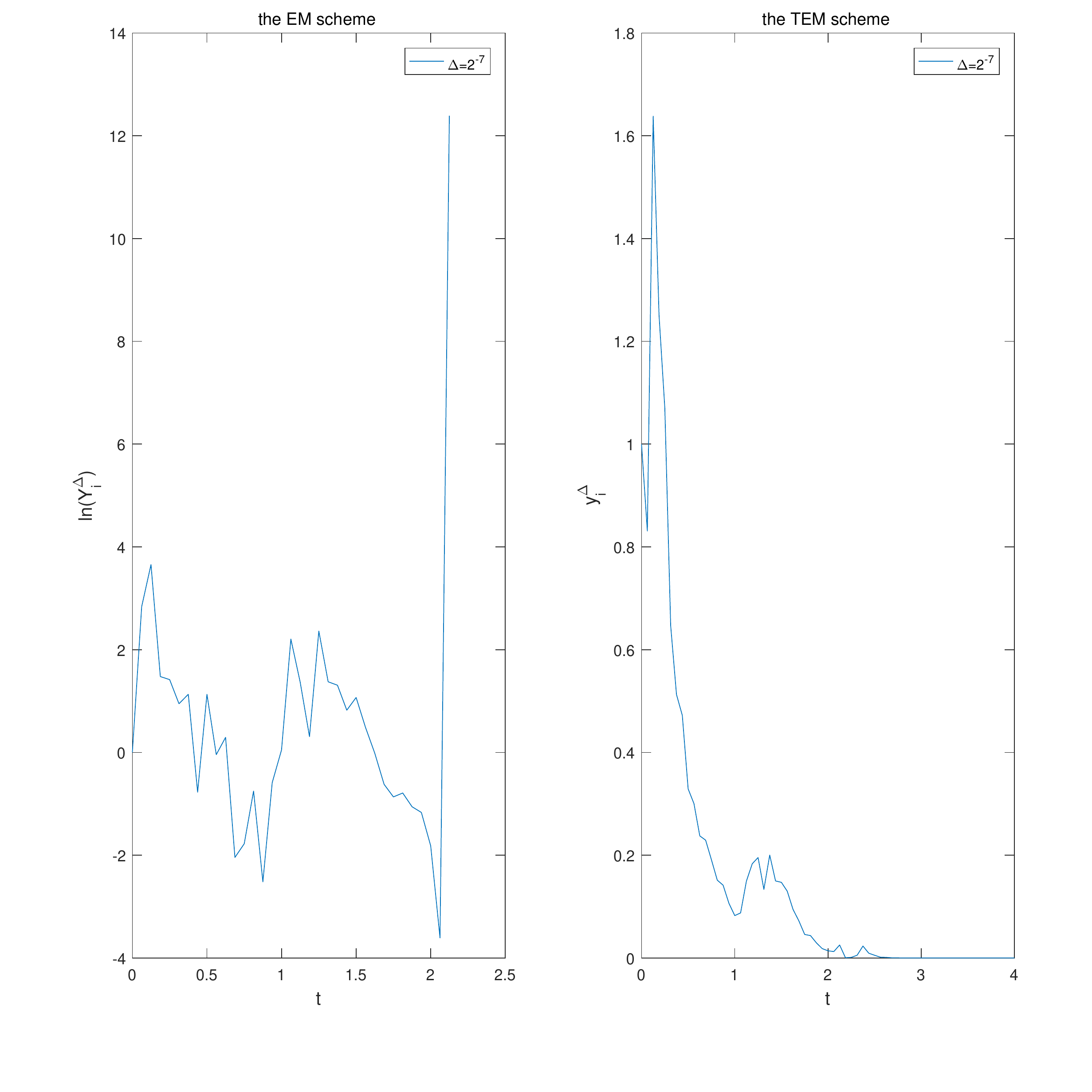}
   \end{center}
  \caption{The sample paths of  $\ln Y_i^{\tr}$ by the EM and
   of the TEM solution $ y_i^{\tr}$ defined in (\ref{s3.28})   
   with $\tr=2^{-7}$ and $t\in[0,~4]$. }
\label{stabplot1}
\end{figure}
}
\end{expl}

\section{Conclusions}\label{Con}

In this paper we construct an explicit numerical scheme under the weakly local Lipschitz condition and  the Khasminskii-type condition, which  numerical solutions are  bounded and converge to the exact solutions in $q$th moment for $q>0$. The $ {1}/{2}$ order convergence rate is obtained for the TEM scheme. Moreover, in order to realize the long time dynamical behavior we propose a more precise TEM scheme.
The exponential stability is kept well by the numerical solutions of the TEM for a large kind of nonlinear SDDEs.

\end{document}